\documentclass[12pt, reqno]{amsart}

\usepackage{amssymb,amsmath,amsthm,graphicx,enumitem,systeme, array,enumitem,commath,mathtools,multicol,url}
\usepackage{setspace,blindtext}
\usepackage{varioref}
\usepackage[colorlinks=true,linkcolor=blue]{hyperref}
\usepackage[capitalise]{cleveref}
\usepackage{blindtext}
\usepackage{fancyhdr}
\usepackage{pgfplots}
\usepackage{geometry}
\usepackage{mathabx}

\usetikzlibrary{calc,angles,quotes}

\usepackage{xcolor}
\DeclareMathOperator{\vol}{vol}
\newcommand{\R}{\mathbb{R}}

\newcommand{\Sph}{\mathbb{S}}

\newcommand{\siglow}{\underline{\sigma}}
\newcommand{\sighigh}{\overline{\sigma}}

\usepackage{caption}
\allowdisplaybreaks 

\theoremstyle{plain}
\newtheorem{theorem}{Theorem}[section]
\newtheorem{lemma}[theorem]{Lemma}

\newtheorem{proposition}[theorem]{Proposition}

\theoremstyle{definition}

\newtheorem{remark}[theorem]{Remark}

\title[Exponential bounds for the spherical Blaschke-Lebesgue problem]{Exponential bounds for the spherical Blaschke-Lebesgue problem}

\author{Abigail Hall}
\address{Department of Mathematics, University of Manitoba, Winnipeg, MB, R3T 2N2, Canada}
 \email{halla6@myumanitoba.ca}

\author{Andriy Prymak}
\address{Department of Mathematics, University of Manitoba, Winnipeg, MB, R3T 2N2, Canada}
 \email{prymak@gmail.com}

\author{Chanatip Sujsuntinukul}
\address{Department of Mathematics, The University of Hong Kong, Pokfulam, Hong Kong}
\email{chanatip@connect.hku.hk}



\begin{document}

\maketitle

\begin{abstract}
The Blaschke-Lebesgue theorem states that the Reuleaux triangle has the smallest area among planar convex bodies of a fixed constant width. We study how small bodies of constant width can be on the unit sphere $\mathbb S^n$ when $n$ is large. For a spherical convex body $K\subset \mathbb S^n$ of constant width $w\in(0,\pi)$, its relative effective radius is \[
  \left(\frac{\mu_n(K)}{\mu_n(\mathbb B^n(w/2))}\right)^{1/n}, \] where $\mu_n$ is the spherical $n$-measure and $\mathbb B^n(w/2)$ is a geodesic ball of radius $w/2$. Let $\sigma_n(w)$ be the infimum of the relative effective radius over all spherical bodies of constant width $w$. Define $\underline{\sigma}(w)=\liminf_{n\to\infty}\sigma_n(w)$ and $\overline{\sigma}(w)=\limsup_{n\to\infty}\sigma_n(w)$. For each fixed $w\in(0,\pi)$, we prove non-trivial bounds \[
  0<\sigma_{\ell}(w)\le \underline{\sigma}(w)\le \overline{\sigma}(w)\le \sigma_u(w)<1, \] where $\sigma_\ell(w)$ and $\sigma_u(w)$ are defined in terms of $w$ either explicitly or through a root of a quartic equation. The upper bounds are obtained by constructing small spherical bodies of constant width: for $w\le\pi/2$ by a spherical version of the recent Arman-Bondarenko-Nazarov-Prymak-Radchenko Euclidean example, and for $w>\pi/2$ by spherical duality. The lower bounds combine a spherical adaptation of Schramm’s illumination argument with a Gaussian autocorrelation method for the associated convex cones. 
  
\end{abstract}

\section{Introduction}

The classical Blaschke-Lebesgue problem asks which convex bodies of a fixed constant width have the least volume.  In the Euclidean plane, Blaschke and Lebesgue independently proved that the area minimizer is the Reuleaux triangle \cite{Blaschke,Lebesgue}.  In higher-dimensional Euclidean space the problem remains open already in dimension three, where the expected minimizers are the Meissner bodies; see, for example, \cite{MMO} for background and further references.  A useful high-dimensional normalization is the effective radius $r(K)$: for a body $K\subset\mathbb R^n$ of constant width $2$, let $r(K)$ be defined by $\vol_n(K)=\vol_n(r(K)\mathbb B^n_E)$, where $\vol_n$ is the Lebesgue measure in $\mathbb{R}^n$ and $\mathbb B^n_E$ is the unit Euclidean ball in $\mathbb R^n$. Further, let $r_n$ be the minimum of $r(K)$ over all bodies of constant width $2$ in $\mathbb R^n$. Evidently, the  isodiametric inequality gives $r_n\le r(\mathbb B^n_E)=1$. In 1988, Schramm proved the non-trivial lower bound
\[
   r_n\ge \sqrt{3+\frac{2}{n+1}}-1,
\]
and asked whether $r_n$ is bounded away from $1$ in all sufficiently large dimensions \cite{SchrammVol}.  This was recently answered affirmatively by Arman, Bondarenko, Nazarov, Prymak and Radchenko: they constructed, for all sufficiently large $n$, Euclidean bodies of constant width $2$ with effective radius less than $0.891$ \cite{original}.  Thus in the Euclidean high-dimensional case there exist bodies of constant width with volume exponentially smaller than that of the ball with the same width, while the exact asymptotic behavior remains unknown and is likely a very hard question to answer.

In this work, we consider the spherical setting. We regard \(\Sph^n\) as the unit sphere in \(\R^{n+1}\), equipped with its geodesic distance and spherical \(n\)-measure \(\mu_n\). Although \(\Sph^n\) is a Riemannian manifold and therefore retains many familiar geometric notions, the spherical setting presents genuinely new features. Let $\mathcal C_n(w)$ denote the class of spherical convex bodies of constant width $w\in(0,\pi)$ in $\Sph^n$.  We use the standard definition by supporting hemispheres and the thickness of the narrowest containing lune; see \Cref{sec:basic-spherical-geometry} for details. The terminology in the literature is not completely uniform: spherical constant width may also be formulated using normal directions, complete bodies, or constant diameter.  The equivalences and distinctions among these viewpoints are surveyed by Lassak~\cite{LM}, and the particular complete-body and constant-diameter equivalences needed here are established in \cite{HW,LM,LM2}.

For $K\in\mathcal C_n(w)$, we define its relative effective radius by
\[
   \operatorname{rer}_{w}(K):=\left(\frac{\mu_n(K)}{\mu_n(\mathbb B^n(w/2))}\right)^{1/n},
\]
where $\mathbb B^n(w/2)$ is any geodesic ball of radius $w/2$ in $\Sph^n$.  We then set
\[
   \sigma_n(w):=\inf_{K\in\mathcal C_n(w)}\operatorname{rer}_w(K),
   \qquad
   \siglow(w):=\liminf_{n\to\infty}\sigma_n(w),
   \qquad
   \sighigh(w):=\limsup_{n\to\infty}\sigma_n(w).
\]
The normalization is chosen so that a ball of radius $w/2$ has relative effective radius $1$. We note that existence of the limit $\lim_{n\to\infty} \sigma_n(w)$ is not known, which is also the case for the Euclidean situation (see \cite[Problem~2]{SchrammVol}).

On $\Sph^2$, the spherical Blaschke-Leichtweiss theorem and its subsequent refinements show that the least-area body of constant width $w$ is a spherical Reuleaux triangle when $0<w\le \pi/2$, whereas, when $\pi/2\le w<\pi$, it is the spherical polar of a triangle of width $\pi-w$ \cite{BezdekBL,BS,FS,Leichtweiss}.
For $w=\pi/2$ this body is the spherical orthant
   $\mathcal{O}_2:=\Sph^2\cap\mathbb{R}_+^3$,
so the exact minimizer has area $\pi/2$ which is $1/8$ of the area of the whole $\Sph^2$. Note that the spherical orthant $\mathcal{O}_n:=\Sph^n\cap\mathbb{R}_+^{n+1}$ belongs to $\mathcal C_n(\pi/2)$ and has the area $1/2^{n+1}$ of the area of the whole $\Sph^n$, so $\sighigh(\pi/2)\le 1/{\sqrt{2}}$. Bezdek~\cite[Conjecture~2.17]{BezdekBL} conjectured that
for every $n\ge3$, the spherical orthant
$\mathcal{O}_n$ minimizes $\mu_n(K)$ among all $K\in\mathcal C_n(\pi/2)$.

Before stating our work, we first define the explicit bases appearing in the bounds. For the lower bound, first let
\begin{align*}
    \sigma_{\ell}^{\text{Sch}}(w):=\begin{cases}
\displaystyle
\frac{2\sqrt{\cos w}\bigl(\sqrt{1+2\cos w}-\sqrt{\cos w}\bigr)}{1+\cos w}
&\text{if }0<w\le\pi/2,\\[2.2ex]
\displaystyle
\frac{2\sqrt{-\cos w}}{1-\cos w}
&\text{if }\pi/2<w<\pi.
\end{cases}
\end{align*}
Next, define 
\[R(w):=\begin{cases}
    1+2\cot w &\text{if $0<w\le\pi/2$}, \\
    1 &\text{if $\pi/2<w<\pi$},
\end{cases}\]
and $D(w):=\sqrt{1+(R(w))^2}$. Let $\xi(w)$ be the smallest positive root of the equation
\[ (R(w))^2\xi^4+\xi^2-2(D(w))\xi+1=0.\]
Then let
\[\sigma_{\ell}^{\text{G}}(w):=\frac{1}{\sin(w/2)}\sin\left(\sin^{-1}(\xi(w))+\max\left\{0, \frac{w-\pi/2}{2}\right\}\right).\]
Finally, put $\sigma_{\ell}(w):=\max\{\sigma_{\ell}^{\text{Sch}}(w), \sigma_{\ell}^{\text{G}}(w)\}$.

For the upper bound, if $0<w\le\pi/2$, set
$
   \theta_-(w):=\cos^{-1}\sqrt{\cos w}$
and let $L_-(w)$ be the supremum of $(\alpha^{2/3}+\beta^{2/3})^{3/2}$ over all $(\alpha,\beta)\in\mathbb R_+^2$ satisfying
\begin{equation*}
   \sqrt{1-\alpha^2-\beta^2}\cos\theta_-(w) - \beta\sin\theta_-(w) \geq \cos w.
\end{equation*}
If $\pi/2<w<\pi$, set
$
   \theta_+(w):=\cos^{-1}\sqrt{-\cos w}$
and
\begin{equation*}
   L_+(w):=
   \sup_{0\le\varphi\le\pi/2}
   \frac{\cot\theta_+(w)\bigl(\cos^{2/3}\varphi+\sin^{2/3}\varphi\bigr)^{3/2}}
        {\sqrt{\cos^2\varphi+\cot^2\theta_+(w)}} .
\end{equation*}
Finally define
\begin{equation*}
\sigma_u(w):=
\begin{cases}
\displaystyle
\frac{L_-(w)}{2\sin(w/2)}
&\text{if } 0<w\le\pi/2,\\[2.2ex]
\displaystyle
\frac{L_+(w)}{2\sin(w/2)}
&\text{if }\pi/2<w<\pi.
\end{cases}
\end{equation*}
The numerical plot of $\sigma_{\ell}$ and $\sigma_u$ is shown in Figure \ref{plot}. We now state our main result.

\begin{theorem}\label{thm:main-intro}
For each fixed $w\in(0,\pi)$, 
we have
\[0<\sigma_{\ell}(w)\le \siglow(w)\le \sighigh(w)\le \sigma_u(w)<1.\]
\end{theorem}


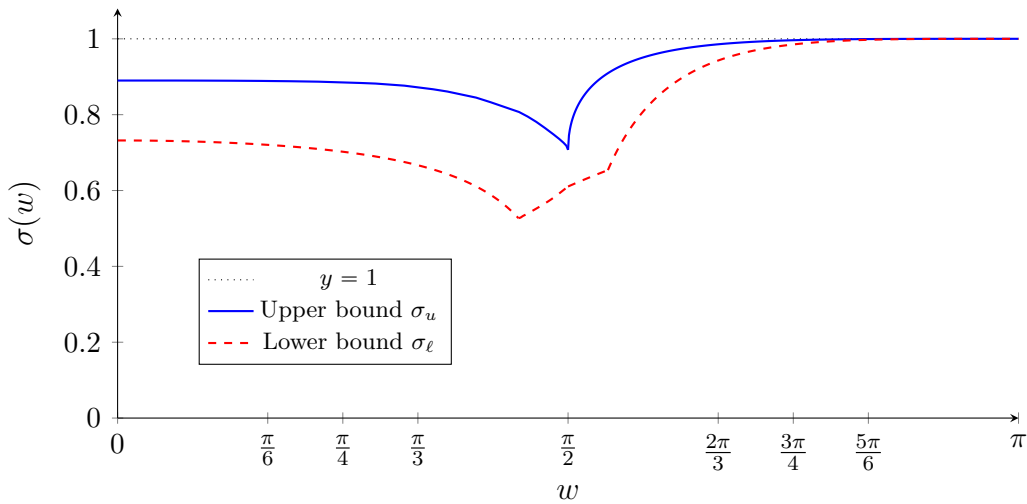
\begin{figure}[ht]
\centering
\begin{tikzpicture}
\pgfmathsetmacro{\wminus}{1.3978350627}      
\pgfmathsetmacro{\whalf}{1.5707963268}       
\pgfmathsetmacro{\wplus}{1.709792577354}     
\pgfmathsetmacro{\betastar}{0.431739420205}
\pgfmathsetmacro{\central}{0.610571743465}   

\begin{axis}[
    width=13.5cm,
    height=7cm,
    xmin=0, xmax=3.1416,
    ymin=0, ymax=1.08,
    axis lines=left,
    xlabel={$w$},
    xlabel style={yshift=-6pt},
    ylabel={$\sigma(w)$},
    xtick={0,0.5236,0.7854,1.0472,1.5708,2.0944,2.3562,2.6180,3.1416},
    xticklabels={$0$,$\frac{\pi}{6}$,$\frac{\pi}{4}$,$\frac{\pi}{3}$,$\frac{\pi}{2}$,$\frac{2\pi}{3}$,$\frac{3\pi}{4}$,$\frac{5\pi}{6}$,$\pi$},
    ytick={0,0.2,0.4,0.6,0.8,1},
    yticklabels={$0$,$0.2$,$0.4$,$0.6$,$0.8$,$1$},
    tick label style={font=\small},
    label style={font=\normalsize},
    legend style={
        draw=black,
        font=\scriptsize,
        at={(0.09,0.13)},
        anchor=south west
    },
    clip=false
]

\addplot[black,dotted,domain=0:3.1416,samples=2] {1};
\addlegendentry{$y=1$}

\addplot[
    blue,
    thick
] coordinates {
    (0.0000,0.8900)
    (0.1000,0.8900)
    (0.2000,0.8900)
    (0.3000,0.8898)
    (0.4000,0.8895)
    (0.5000,0.8890)
    (0.5236,0.8888)
    (0.6000,0.8882)
    (0.7000,0.8870)
    (0.7854,0.8850)
    (0.8500,0.8835)
    (0.9000,0.8820)
    (1.0000,0.8765)
    (1.0472,0.8720)
    (1.1000,0.8670)
    (1.1500,0.8610)
    (1.2000,0.8530)
    (1.2500,0.8450)
    (1.3000,0.8330)
    (1.3500,0.8200)
    (1.4000,0.8070)
    (1.4300,0.7950)
    (1.4600,0.7810)
    (1.4900,0.7650)
    (1.5200,0.7480)
    (1.5400,0.7360)
    (1.5550,0.7260)
    (1.5650,0.7180)
    (1.5708,0.7071)
};
\addlegendentry{Upper bound $\sigma_u$}

\addplot[
    blue,
    thick,
    domain=\whalf:3.1416,
    samples=500,
    forget plot
]
{(1 + sqrt(-cos(deg(x))))/sqrt(2*(1 - cos(deg(x))))};

\addplot[
    red,
    thick,
    dashed,
    domain=0.001:\wminus,
    samples=400
]
{2*sqrt(cos(deg(x)))*
 (sqrt(1+2*cos(deg(x))) - sqrt(cos(deg(x))))
 /(1+cos(deg(x)))};
\addlegendentry{Lower bound $\sigma_\ell$}

\addplot[
    red,
    thick,
    dashed,
    forget plot
] coordinates {
    (1.39783506,0.52705)
    (1.40499,0.52854)
    (1.41372,0.53298)
    (1.42245,0.53631)
    (1.43117,0.53990)
    (1.43990,0.54357)
    (1.44862,0.54735)
    (1.45735,0.55115)
    (1.46608,0.55498)
    (1.47480,0.55890)
    (1.48353,0.56289)
    (1.49226,0.56698)
    (1.50098,0.57116)
    (1.50971,0.57554)
    (1.51844,0.58008)
    (1.52716,0.58488)
    (1.53589,0.58990)
    (1.54462,0.59504)
    (1.55334,0.60027)
    (1.56207,0.60546)
    (1.57080,0.61057)
};

\addplot[
    red,
    thick,
    dashed,
    domain=\whalf:\wplus,
    samples=300,
    forget plot
]
{(
   \betastar*cos(deg((x-\whalf)/2))
   + sqrt(1-\betastar^2)*sin(deg((x-\whalf)/2))
 )/sin(deg(x/2))};

\addplot[
    red,
    thick,
    dashed,
    domain=\wplus:3.1415,
    samples=350,
    forget plot
]
{2*sqrt(-cos(deg(x)))/(1-cos(deg(x)))};

\end{axis}
\end{tikzpicture}
\caption{Graph of the lower and upper effective-radius bases
$\sigma_\ell$ and $\sigma_u$.}
\label{plot}
\end{figure}

The proof is divided into two upper-bound constructions and two different lower-bound arguments. Regarding the upper bound, for widths below $\pi/2$, we construct a spherical analogue of the Euclidean body from \cite{original} (\Cref{sec:upper-bound-small-width-case}), while width above $\pi/2$ are treated by spherical polarity (\Cref{sec:upper-bound-large-width-case}). Explicit volume computation and asymptotic estimation lead to optimization problems. 
For illustration, examples of how our body $M$ looks on $\mathbb{S}^2$ are shown in Figure \ref{fig:body}. The colors represent different orthants\footnote{A three-dimensional model is available at \url{http://prymak.net/spherical.html}.}.

For the lower bound, $\sigma_{\ell}^{\text{Sch}}$ is the result of a spherical adaptation of Schramm's illumination argument, using a ball-containment lemma, spherical Jung-type radius estimates, and the same exponential cap asymptotics (\Cref{sec:lower}). In the case $w=\pi/2$, such an approach produces only a subexponential estimate, see 
Bezdek's work~\cite[Remark~2.16]{BezdekBL}. We then present a second method yielding $\sigma_{\ell}^{\text{G}}$ which is based on the Gaussian measure, autocorrelation of the associated convex cone, and Moreau decomposition, and provides the stronger estimate near $w=\pi/2$ (\Cref{sec:lower_new}).



\begin{figure} \centering   \includegraphics[width=9cm]{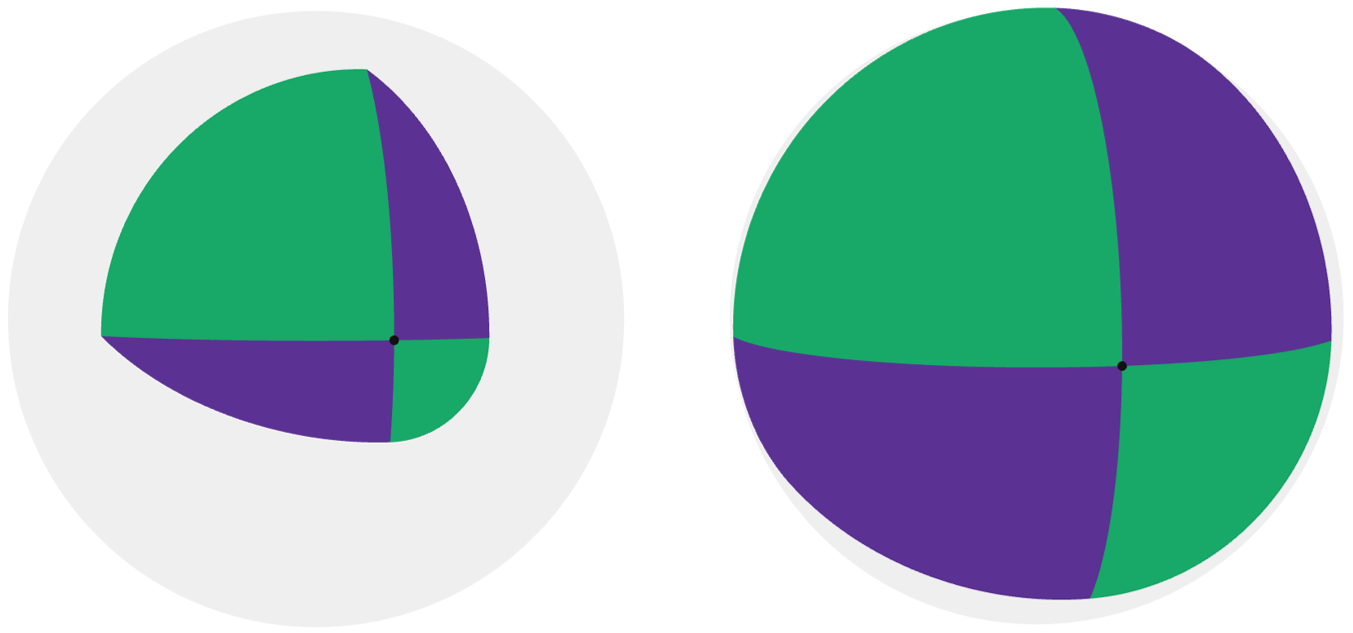}    \caption{Illustrations of $M$ for $w=\pi/3$ (left) and $w=2\pi/3$ (right).}    \label{fig:body}\end{figure}


\section{Basic spherical geometry}\label{sec:basic-spherical-geometry}

We briefly recall the spherical geometric notions used throughout the paper. For more details, we refer the reader to \cite{LM}. We regard $\mathbb{S}^n$ as the unit sphere in $\mathbb{R}^{n+1}$, endowed with the geodesic distance $d(x,y):=\cos^{-1}\langle x,y\rangle$ where $\langle\cdot, \cdot\rangle$ is the usual Euclidean dot product. For $x\in\mathbb{S}^n$, let \[H_x:=\{y\in\mathbb{S}^n:\langle x,y\rangle\geq 0\}\] be the closed hemisphere centered at $x$; its boundary $\partial H_x$ is the corresponding great subsphere. For $x\in\mathbb{S}^n$ and $0<r<\pi$, denote the geodesic ball by
\[
\mathbb{B}^n(x,r):=\{y\in\mathbb{S}^n:d(x,y)\leq r\}
=\{y\in\mathbb{S}^n:\langle x,y\rangle\geq \cos r\}.
\]
For $1\le j\le n+1$, denote the vector where the $j$th component is $1$ and the others are $0$ by $e_j$.
We shall usually fix the north pole $p=e_1$ and abbreviate $\mathbb{B}^n(r):=\mathbb{B}^n(p,r)$. Its spherical measure is $\mu_n(\mathbb{B}^n(r))=\omega_n s_n(r)$, where $s_n(r):=\int_0^r\sin^{n-1}t\,dt$ and $\omega_n=2\pi^{n/2}/\Gamma(n/2)$ is the surface area of $\mathbb{S}^{n-1}$; see, for example, \cite{KB}.

If $x,y\in\mathbb{S}^n$ are not antipodal, we denote by $[x,y]$ the unique geodesic segment joining them. A set $K\subset\mathbb{S}^n$ is spherically convex if it contains no pair of antipodal points and $[x,y]\subset K$ for every $x,y\in K$. In particular, every closed spherical ball of radius $<\pi/2$ is spherically convex, and intersections of spherically convex sets are again spherically convex.

For $x,y\in\mathbb{S}^n$, the intersection $L(x,y):=H_x\cap H_y$ is called a spherical lune. Its thickness is $\Delta(L(x,y)):=\pi-d(x,y)$, which is the geodesic distance between the two bounding great subspheres $\partial H_x$ and $\partial H_y$.
Let $K\subset\mathbb{S}^n$ be a spherical convex body. A hemisphere $H_x$ is said to support $K$ if $K\subset H_x$ and $K\cap\partial H_x\neq\emptyset$. Given such a supporting hemisphere $H_x$, the width of $K$ determined by $H_x$ is defined as
\[
\inf\{\Delta(L(x,y)):K\subset H_x\cap H_y\}.
\]
We say that $K$ has constant width $w\in(0,\pi)$ if this quantity is equal to $w$ for every supporting hemisphere $H_x$. 

In addition, write $\mathrm{diam}(K):=\sup\{d(x,y):x,y\in K\}$. The body $K$ is said to have constant diameter $\delta\in(0, \pi)$ if $\mathrm{diam}(K)=\delta$ and, for every $x\in\partial K$, there exists $y\in\partial K$ such that $d(x,y)=\delta$. Finally, we say that $K$ is a complete set of diameter $\delta\in(0, \pi/2]$ if $\mathrm{diam}(K\cup\{x\})>\delta$ for any $x\not\in K$. This is equivalent to
\[K=\bigcap_{y\in K}\mathbb{B}^n(y, \delta).\]

\section{Upper bounds below the central width}\label{sec:upper-bound-small-width-case}

\subsection{Construction via completion}

Let $n\ge 2$ and $0<w\le\pi/2$. Define \[T_1:=\mathbb{S}^n\cap\mathbb{R}_+^{n+1}\quad\text{and}\quad T_2:=\mathbb{S}^n\cap(\mathbb{R}_+\times\mathbb{R}_-^n).\] Consequently, let \[L_1:=\{x\in T_1 : d(x, p)=\theta_-\}\quad \text{and}\quad L_2:=\{x\in T_2 : d(x, p)=w-\theta_-\}\] where $\theta_-:=\cos^{-1}\sqrt{\cos w}$. Finally let $L:=L_1\cup L_2$ and
\begin{align*}
    M_w:=\bigcap_{x\in L}\mathbb{B}^n(x, w).
\end{align*}
For brevity, we will usually omit the subscript $w$ when the context is clear. Note that $M=\mathcal{O}_n$ at $w=\pi/2$. Throughout this section, assume $0<w<\pi/2$ when proving results. 
It is easy to check that $w/2\le \theta_-\le w$ and $\mathrm{diam}(L_1)=\mathrm{diam}(L)=w$. Also, we know the boundary cases $M\cap T_1=\mathbb{B}^n(\theta_-)\cap T_1$ and $M\cap T_2=\mathbb{B}^n(w-\theta_-)\cap T_2$.
Our next task is to provide an analytic description for $M$.

\begin{lemma}\label{constraintlemma}
    Suppose $v=(v_1, v_2, \ldots, v_{n+1})\in \mathbb{S}^n$ with $x:=v_1$. We decompose the tangential part $v_\perp:=(v_2,\dots,v_{n+1})$ into positive and negative orthant components:
    \[v^{\pm}:=(\max\{\pm v_2, 0\}, \ldots, \max\{\pm v_{n+1}, 0\}).\]
    Then $v_{\perp}=v^+-v^-$ and $\langle v^+, v^-\rangle =0$. Let $\alpha=\|v^+\|$ and $\beta=\|v^-\|$. 
    Label the inequality constraints
    \begin{align}
        \sqrt{1-\alpha^2-\beta^2}\cos\theta_- - \beta\sin\theta_- &\geq \cos w, \label{one} \\
        \sqrt{1-\alpha^2-\beta^2}\cos(w-\theta_-) - \alpha\sin(w-\theta_-) &\geq \cos w. \label{two}
    \end{align}
    Then $v\in M$ if and only if $x\ge 0$ and $(\alpha, \beta)$ lies in the region
    \begin{align*}
    \mathcal{A}_w := \left\{ (\alpha, \beta) \in \mathbb{R}_+^2 : \text{ $(\alpha,\beta)$ satisfy \eqref{one} and \eqref{two}} \right\}.
    \end{align*}
\end{lemma}
\begin{proof}
First note that, since $x\ge 0$ on $M\subset\mathbb{S}^n$, $x=\sqrt{1-\alpha^2-\beta^2}$.\\
   ($\Rightarrow$) For now, let us assume $\alpha, \beta\ne 0$. Suppose $v\in M$. For any $y\in L_1$, we can write $y=(\cos\theta_-, (\sin\theta_-)u)$ with $u\in\mathbb{S}^{n-1}\cap\mathbb{R}^n_+$. The inner product is
    \[
        \langle v, y\rangle = x\cos\theta_- + \sin\theta_-\,\langle v_\perp, u\rangle 
        = x\cos\theta_- + \sin\theta_-\big(\langle v^+, u\rangle - \langle v^-, u\rangle\big).
    \]
    Since $u, v^+ \in \mathbb{R}^n_+$, we have $\langle v^+, u\rangle \ge 0$. By Cauchy--Schwarz,  $\langle v^-, u\rangle \le \|v^-\|\|u\| = \beta$, with equality when $u$ aligns with $v^-/\beta$. Thus
    \[
        \min_{y\in L_1} \langle v, y\rangle = x\cos\theta_- - \beta\sin\theta_-.
    \]
    The condition $v\in M$ requires $\langle v, y\rangle\ge\cos w$ for all $y\in L_1$, so the minimum satisfies \eqref{one}. 
    
    For $y\in L_2$, write $y=(\cos(w-\theta_-), (\sin(w-\theta_-))u)$ with $u\in\mathbb{S}^{n-1}\cap\mathbb{R}^n_-$. Now observe that $\langle v^+, u\rangle \le 0$ and $\langle v^-, u\rangle \le 0$. The minimum of $\langle v_\perp, u\rangle$ occurs when $u$ aligns with $-v^+/\alpha$, giving $\min \langle v_\perp, u\rangle = -\alpha$. Hence
    \[
        \min_{y\in L_2} \langle v, y\rangle = x\cos(w-\theta_-) - \alpha\sin(w-\theta_-) \geq \cos w,
    \]
    which is \eqref{two}. Therefore $(\alpha,\beta)\in\mathcal{A}_w$.

    Finally, suppose $\alpha=0$ or $\beta=0$. Let us only prove the case $\beta=0$. Let $v\in M$, then $v_{\perp}=v^+$ and $\|v_{\perp}\|=\alpha$. So we have \begin{align}\label{boundarycase}
        \sqrt{1-\alpha^2}=\langle p, v\rangle=\cos d(p, v)\ge\cos \theta_-=\sqrt{\cos w}.
    \end{align} Multiplying both sides of \eqref{boundarycase} by $\cos\theta_-$, we are done.

    ($\Leftarrow$) Conversely, suppose $(\alpha,\beta)\in\mathcal{A}_w$. By the same minimization argument, we have $\min_{y\in L_1} \langle v, y\rangle \geq \cos w$ and $\min_{y\in L_2} \langle v, y\rangle \geq \cos w$. Thus $\langle v, y\rangle \geq \cos w$ for all $y\in L_1\cup L_2$, which by definition implies $v\in M$.
\end{proof}

We now show that with our choice of $\theta_-$ one can drop~\eqref{two} from Lemma~\ref{constraintlemma}, i.e., in a sense the part $L_1$ is already defining $\mathcal{A}_w$.
\begin{proposition}\label{cond2redundant} For $(\alpha, \beta)\in\mathbb{R}_+^2$, 
\eqref{one} implies \eqref{two}.
\end{proposition}
\begin{proof}
    Suppose $(\alpha, \beta)\in\mathbb{R}^2_+$ satisfies \eqref{one}. Then
\begin{align*}
    x\cos\theta_- \ge\cos w+\beta\sin\theta_- \ge\cos w\quad \Rightarrow\quad x\ge \frac{\cos w}{\cos\theta_-
    }=\cos\theta_-.
\end{align*}
Note that $\alpha=\sqrt{1-x^2-\beta^2}\le\sqrt{1-x^2}$. Because $\sin(w-\theta_-)\ge 0$,
\begin{align}\label{trigineq}
     x\cos(w-\theta_-) - \alpha\sin(w-\theta_-) \ge x\cos(w-\theta_-) - \sqrt{1-x^2}\sin(w-\theta_-).
\end{align}
 Let $\phi := \cos^{-1} x \in [0, \pi/2]$. Then $x=\cos\phi$ and $\sqrt{1-x^2}=\sin\phi$, so the right-hand side of \eqref{trigineq} is $\cos(\phi + w - \theta_-)$.
Since $x \ge \cos\theta_-$, we have $\phi \le \theta_-$, hence $\phi + w - \theta_- \le w$. As cosine is decreasing on $[0, \pi/2]$ and $w \le \pi/2$, we yield
    $\cos(\phi + w - \theta_-) \ge \cos w$.
    Combining the inequalities yields \eqref{two} as desired. 
\end{proof}

After some algebraic manipulations, we can rewrite \eqref{one} as
\begin{align}\label{ellipse}
    \frac{\alpha^2}{\sin^2w}+\frac{(\beta+\cos w\sin\theta_-)^2}{\cos^2\theta_-\sin^2w}\le 1.
\end{align}Thus $\mathcal{A}_w$ is the positive region bounded by an ellipse.

\begin{proposition}\label{constantwidth}
    $M$ is a spherical body of constant width $w$.
\end{proposition}
\begin{proof}
    First it is easy to see that $\mathrm{diam}(M)\ge w$ since $L\subset M$ and $\mathrm{diam}(L)=w$. Now assume we know that $\mathrm{diam}(M)\le w$. Then we see that
    \[M\subset\bigcap_{x\in M}{\mathbb{B}}^n(x, w)\subset \bigcap_{x\in L}{\mathbb{B}}^n(x, w)=M.\]
    So $M$ is a complete set of diameter $w$. Using \cite{LM2}, we can conclude that $M$ is a body of constant width $w$.

    Now let us justify why $\mathrm{diam}(M)\le w$.
Take any $u, v \in M$ with parameters $(\alpha, \beta)$ and $(\alpha', \beta')$ in $\mathcal{A}_w$. Decomposing $u_{\perp}=u^+-u^-$, $v_{\perp}=v^+-v^-$, and using Cauchy--Schwarz, 
    \begin{align}\label{cauc}
        \langle u, v\rangle=xx'+\langle u^+, v^+\rangle+\langle u^-, v^-\rangle-\langle u^+, v^-\rangle-\langle u^-, v^+\rangle
        \geq xx' - \alpha\beta' - \beta\alpha',
    \end{align}
    where $x=\sqrt{1-\alpha^2-\beta^2}$ and $x'=\sqrt{1-(\alpha')^2-(\beta')^2}$.
    By \eqref{one} and $\cos^2\theta_-=\cos w$, we find
    \begin{align*}
        x \ge \beta\sqrt{\frac{1-\cos w}{\cos w}}+\sqrt{\cos w}\quad\text{and}\quad x' \ge \beta'\sqrt{\frac{1-\cos w}{\cos w}}+\sqrt{\cos w}.
    \end{align*}
    Hence, from \eqref{cauc}, we have
    \begin{align*}
        \langle u, v\rangle&\ge \cos w+\beta\beta'\left(\frac{1-\cos w}{\cos w}\right)+\beta(\sqrt{1-\cos w}-\alpha')+\beta'(\sqrt{1-\cos w}-\alpha).
    \end{align*}
    Finally note that $\max_{(\alpha, \beta)\in\mathcal{A}_w}\alpha$ is the positive $\alpha$-intercept of $\mathcal{A}_w$ which is $\sqrt{1-\cos w}$. Therefore $\langle u, v\rangle\ge\cos w$ as desired.
\end{proof}

\subsection{Volume setup}

Our next goal is to compute $\mu_n(M)$.  We use the standard Riemannian volume formula in which the volume density is given by the square root of the determinant of the Gram matrix; see, for example, Chapter~1 of \cite{Jost}.

First, we decompose $M$ according to the sign patterns of the tangential coordinates $v_2,\dots,v_{n+1}$.  (The first coordinate $x=v_1$ is non-negative on $M$ and thus does not affect the sign combinatorics.) A $(k, n-k)$ orthant $Q$ consists of points with exactly $k$ positive and $n-k$ negative tangential coordinates, while the first coordinate is strictly positive. The boundary cases $(n,0)$ and $(0,n)$ contribute to $\frac{1}{2^n}\mu_n(\mathbb{B}^n(\theta_-))$ and $\frac{1}{2^n}\mu_n(\mathbb{B}^n(w-\theta_-))$ respectively.

Now let us focus on $1\le k\le n-1$.
Fix a $(k, n-k)$ orthant $Q$. After reordering coordinates, we parametrize $M\cap Q$ by
\[
\Phi_Q(\alpha, \beta, s, t)=\big(\sqrt{1-\alpha^2-\beta^2},\; \alpha\psi_1(s),\; -\beta\psi_2(t)\big),
\]
where $(\alpha, \beta)\in\mathcal{ A}_w$, $s\in\mathcal{D}_1:=\mathbb{B}_E^{k-1}\cap \mathbb{R}^{k-1}_+$, $t\in\mathcal{D}_2:=\mathbb{B}_E^{n-k-1}\cap \mathbb{R}^{n-k-1}_+$, and $\psi_1,\psi_2$ are the standard graph parametrizations of the positive spherical wedges $\mathbb{S}^{k-1}\cap\mathbb{R}^k_+$ and $\mathbb{S}^{n-k-1}\cap\mathbb{R}^{n-k}_+$, respectively. For instance, \[\psi_1(s)=\psi_1(s_1, \ldots, s_{k-1})=\left(s_1, \ldots, s_{k-1}, \sqrt{1-\|s\|^2}\right).\]
The Gram matrix of $\Phi_Q$ is block-diagonal:
\[
G(\Phi_Q)=\begin{pmatrix}
A & O & O \\
O & \alpha^2G(\psi_1) & O \\
O & O & \beta^2G(\psi_2)
\end{pmatrix},
\quad
A:=\begin{pmatrix}
\frac{\alpha^2}{1-\alpha^2-\beta^2}+1 & \frac{\alpha\beta}{1-\alpha^2-\beta^2}\\
\frac{\alpha\beta}{1-\alpha^2-\beta^2} & \frac{\beta^2}{1-\alpha^2-\beta^2}+1
\end{pmatrix}.
\]
A direct computation gives $\det A = (1-\alpha^2-\beta^2)^{-1}$. The volume form is therefore
\begin{align*}
    dV &=\sqrt{\det G(\Phi_Q)}\,ds\,dt\,d\alpha \,d\beta\\
    &=\sqrt{(\det A)(\det \alpha^2G(\psi_1))(\det\beta^2 G(\psi_2)) }ds\,dt\,d\alpha \,d\beta
    \\ &=\frac{\alpha^{k-1}\beta^{n-k-1}}{\sqrt{1-\alpha^2-\beta^2}} \cdot \frac{ds}{\sqrt{1-\|s\|^2}} \cdot \frac{dt}{\sqrt{1-\|t\|^2}} \, d\alpha\, d\beta.
\end{align*}
Integrating over the angular domains $\mathcal{D}_1,\mathcal{D}_2$ yields the surface areas of the positive wedges:
\[
\int_{\mathcal{D}_1}\frac{ds}{\sqrt{1-\|s\|^2}} = \frac{\omega_k}{2^k} \quad\text{and}\quad
\int_{\mathcal{D}_2}\frac{dt}{\sqrt{1-\|t\|^2}} = \frac{\omega_{n-k}}{2^{n-k}}.
\]
By Fubini--Tonelli's theorem, we obtain
\[
\mu_n(M\cap Q)=\frac{\omega_k \omega_{n-k}}{2^n}\iint_{\mathcal{A}_w}\frac{\alpha^{k-1}\beta^{n-k-1}}{\sqrt{1-\alpha^2-\beta^2}}\,d\alpha\,d\beta.
\]
Since there are $\binom{n}{k}$ orthants of type $(k,n-k)$, combining all these formulas, we yield $\mu_n(M)=V_1+V_2$
where
\[V_1:=\frac{\omega_n}{2^n}\left(\int_0^{\theta_-}\sin^{n-1}xdx+\int_0^{w-\theta_-}\sin^{n-1}xdx\right)\]
and
\[V_2:=\sum_{k=1}^{n-1}\binom{n}{k}\frac{\omega_k \omega_{n-k}}{2^n}\iint_{\mathcal{A}_w}\frac{\alpha^{k-1}\beta^{n-k-1}}{\sqrt{1-\alpha^2-\beta^2}}\,d\alpha\,d\beta.\]

\subsection{Effective radius computation}

Before proceeding, we introduce asymptotic notations. For two sequences of positive real numbers $\{a_n\}$ and $\{b_n\}$, we write $a_n\sim b_n$
and $a_n\lesssim b_n$ to mean, respectively,
\[\lim_{n\to\infty}\frac{a_n}{b_n}=1\quad\text{and}\quad 0\leq\limsup_{n\to\infty}\frac{a_n}{b_n}\leq 1.\]
Two key asymptotic facts we will use here are $\lim_{n \to \infty} (n^c)^{1/n} = 1$ for any fixed $c\in\mathbb{R}$, and  $(\sum_{i=1}^m a_i)^{1/n} \sim (\max a_i)^{1/n}$ for fixed $m$, or even if $m=n$.

Next we obtain a good asymptotic upper bound for $\mathrm{rer}_w(M)$. Clearly,
\begin{align*}
    \left(\frac{\mu_n(M)}{\mu_n(\mathbb{B}^n(w/2))}\right)^{1/n}=\left(\frac{V_1+V_2}{\mu_n(\mathbb{B}^n(w/2))}\right)^{1/n}\sim\left(\max\left\{\frac{V_1}{\mu_n(\mathbb{B}^n(w/2))}, \frac{V_2}{\mu_n(\mathbb{B}^n(w/2))}\right\}\right)^{1/n}.
\end{align*}
Next we estimate each term. Recall that 
\begin{align}\label{intsin}
    \left(\int_0^r\sin^nxdx\right)^{1/n}\sim \sin r,
\end{align}
where $0<r<\pi/2$, which follows from the $L^p$-to-$L^{\infty}$-limit. 

Now, for the first term, by \eqref{intsin}, we see that
\begin{align*}
    \left(\frac{V_1}{\mu_n(\mathbb{B}^n(w/2))}\right)^{1/n}=\frac{1}{2}\left(\frac{s_{n}(\theta_-)+s_n(w-\theta_-)}{s_{n}(w/2)}\right)^{1/n}\sim\frac{1}{2}\left(\frac{s_{n}(\theta_-)}{s_{n}(w/2)}\right)^{1/n}\sim \frac{1}{2}\cdot\frac{\sin\theta_-}{\sin (w/2)}.
\end{align*}
For the second term, first note that for $(\alpha, \beta)\in\mathcal{A}_w$, by \eqref{one}, we have
\begin{align*}
    \frac{1}{\sqrt{1-\alpha^2-\beta^2}} &\le\frac{\cos\theta_-}{\cos w}\le\frac{\cos (w/2)}{\cos w}=:C_w.
\end{align*}
So we obtain the inequality bound
\begin{align}\label{seven}
    V_2\le C_w\sum_{k=1}^{n-1}\binom{n}{k}\frac{\omega_k\omega_{n-k}}{2^n}\iint_{\mathcal{A}_w}\alpha^{k-1}\beta^{n-k-1}d\alpha d\beta.
\end{align}
Next using Stirling's formula, one can show that, uniformly over $0\le k\le n$,
\begin{align*}
\left(\frac{\omega_k\omega_{n-k}}{\omega_n}\right)^{1/n} &\sim \left(\frac{k(n-k)}{n}\cdot\frac{\sqrt{\pi n}}{\pi\sqrt{k(n-k)}}\cdot\frac{n^{n/2}}{k^{k/2}(n-k)^{(n-k)/2}}\right)^{1/n}\sim\left(\sqrt{\binom{n}{k}}\right)^{1/n}.
\end{align*}
Combining this with \eqref{intsin}, \eqref{seven}, and the fact that $x\mapsto x^{3/2}$ is superadditive, 
\begin{align*}
    \left(\frac{V_2}{\mu_n(\mathbb{B}^n(w/2))}\right)^{1/n} &\lesssim  \frac{1}{2\sin(w/2)}\left(\sum_{k=1}^{n-1}\binom{n}{k}\frac{\omega_k\omega_{n-k}}{\omega_{n}}\iint_{\mathcal{A}_w}\alpha^{k-1}\beta^{n-k-1}\,d\alpha\,d\beta\right)^{1/n} \\
    &\sim \frac{1}{2\sin(w/2)}\left(\max_{1\le k\le n-1}\binom{n}{k}^{3/2}\iint_{\mathcal{A}_w}\alpha^{k-1}\beta^{n-k-1}\,d\alpha\,d\beta\right)^{1/n}\\
    &\lesssim  \frac{1}{2\sin(w/2)}\left(\sum_{k=0}^n\binom{n}{k}^{3/2}\mathrm{Area}(\mathcal{A}_w)\sup_{(\alpha, \beta)\in\mathcal{A}_w}\alpha^k\beta^{n-k}\right)^{1/n}\\
    &\lesssim \frac{1}{2\sin(w/2)}\left(\sum_{k=0}^n\binom{n}{k}\sup_{(\alpha, \beta)\in\mathcal{A}_w}(\alpha^{2k/3})(\beta^{2(n-k)/3})\right)^{3/2n}\\
    &\lesssim\frac{1}{2\sin(w/2)}\sup_{(\alpha, \beta)\in\mathcal{A}_w}(\alpha^{2/3}+\beta^{2/3})^{3/2}.
\end{align*}

Denote
$L_-(w):=\sup_{(\alpha, \beta)\in\mathcal{A}_w}(\alpha^{2/3}+\beta^{2/3})^{3/2}$.
Therefore we deduce that
\begin{align}\label{deducesigup}
    \sigma_n(w)\le\left(\frac{\mu_n(M)}{\mu_n(\mathbb{B}^n(w/2))}\right)^{1/n}\lesssim\frac{1}{2\sin(w/2)}\max\{\sin\theta_-, L_-(w)\}=:\sigma_u(w).
\end{align}
Notice that $\frac{\sin\theta_-}{2\sin(w/2)}=\frac{1}{\sqrt{2}}<1$. 
Now we describe $L_-(w)$.

\begin{proposition} \label{uniqueroot}
Let $0<w<\pi/2$.
Denote \(\Lambda(w):=\cos\theta_-\). Then $L_-(w)$ is attained at
\[
(\alpha_*, \beta_*)=\left(\sqrt{(1-\Lambda^4)(1-(\tau(w))^2)},
 \sqrt{1-\Lambda^4}
\left(
\Lambda\tau(w)-\frac{\Lambda^2}{\sqrt{1+\Lambda^2}}
\right)\right),
\]
where $ \tau(w)$ is the unique real root
of the polynomial
\[ k_w(x):= (1-\Lambda^2)x^4 -\frac{\Lambda}{\sqrt{1+\Lambda^2}}x^3 +2\Lambda^2x^2 -\Lambda^2\] 
in the interval $\left[\frac{\Lambda}{\sqrt{1+\Lambda^2}},1\right]$.
\end{proposition} 
\begin{proof} Since \(x\mapsto x^{2/3}\) is increasing, the maximum is attained on the boundary ellipse of $\mathcal{A}_w$, described in \eqref{ellipse}. We parametrize this boundary by \begin{align*}
    \alpha(\varphi) &=\sin w\cos \varphi=\sqrt{1-\Lambda^4}\cos\varphi,\\ 
    \beta(\varphi)&=-\cos w\sin\theta_-+\cos\theta_-\sin w\sin \varphi=-\Lambda^2\sqrt{1-\Lambda^2}+\Lambda\sqrt{1-\Lambda^4}\sin\varphi, 
\end{align*} where $ \varphi\in\left[\sin^{-1}\frac{\Lambda}{\sqrt{1+\Lambda^2}},\frac{\pi}{2}\right]$. Set \(x=\sin \varphi\). Note that $\frac{\Lambda}{\sqrt{1+\Lambda^2}}\le x\le1$, and a direct computation gives 
\begin{align*}
    f_w(x) &:= ((\alpha(\sin^{-1}x))^{2/3}+(\beta(\sin^{-1}x))^{2/3})^{3/2}
\\
&= \sqrt{1-\Lambda^4} \left( (1-x^2)^{1/3} + \left(\Lambda x-\frac{\Lambda^2}{\sqrt{1+\Lambda^2}}\right)^{2/3} \right)^{3/2}.
\end{align*} It is enough to maximize $g_w(x):=\left(\frac{f_w(x)}{\sqrt{1-\Lambda^4}}\right)^{2/3}$. Differentiating, we see that \(g_w'(x)=0\) is equivalent to \(k_w(x)=0\). Moreover, \(g_w'\) has the opposite sign to \(k_w\). Now $ k_w\left(\frac{\Lambda}{\sqrt{1+\Lambda^2}}\right)<0$ and $k_w(1)>0$.
This implies the existence of $\tau$ by the intermediate value theorem.
It remains to show uniqueness. We have \[ k_w'(x) = x\left( 4(1-\Lambda^2)x^2 -\frac{3\Lambda}{\sqrt{1+\Lambda^2}}x +4\Lambda^2 \right). \] The quadratic factor is positive on \([0,1]\). Hence \(k_w\) is strictly increasing on the relevant interval. Thus \(k_w\) has a unique root \(\tau(w)\), and the first derivative test shows that the maximum is attained there. \end{proof}

\begin{remark}
Substituting \((\alpha,\beta)=\sin(w/2)(x,y)\), where $x, y\ge 0$, into \eqref{ellipse} gives
\begin{align}\label{normalise}
    \frac{x^2}{2(1+\cos w)}+\frac{(y+\sqrt{2}\cos w)^2}{2\cos w(1+\cos w)}\le 1.
\end{align}
In particular, as $w\to 0^+$, this becomes
$x^2+(y+\sqrt2)^2\le4$.
Thus the rescaled spherical constraint recovers the Euclidean constraint in
\cite{original}. On the other hand, as \(w\to\pi/2^-\), we have
\(\Lambda=\sqrt{\cos w}\to0\), so \(k_w(x)\to x^4\) and the critical root satisfies
\(\tau(w)\to0\). Hence
$
\lim_{w\to\pi/2^-}\sigma_u(w)=\frac1{\sqrt2}$,
which is consistent with the construction degenerating to an orthant.
\end{remark}

Since the formula for $\sigma_u$ is not explicit, it is not obvious that this quantity is indeed strictly less than $1$. We will then need to verify this. In fact, we can show something stronger.

\begin{proposition}
    For $w\in(0, \pi/2)$, define \[\mathcal{R}_w:=\{(x, y)\in\mathbb{R}^2_+ : (x, y)\text{ satisfies \eqref{normalise}}\}.\]
    If $0<{w}'\le w<\pi/2$, then $\mathcal{R}_w\subset \mathcal{R}_{{w}'}$. Hence $\sigma_u(w)$ is decreasing over $(0, \pi/2)$.
\end{proposition}
\begin{proof}
    Let $(x, y)\in \mathcal{R}_{w}$. Then by definition, this is equivalent to
    \[G_w(x, y):=(\cos w)x^2+y^2+2\sqrt{2}(\cos w)y-2\cos w\le 0. \]
    So we have
    \begin{align*}
        G_{w'}(x, y)=G_{w}(x, y)+(\cos{w}'-\cos w)(x^2+2\sqrt{2}y-2)\le G_w(x, y)\le 0,
    \end{align*}
    implying $(x, y)\in\mathcal{R}_{{w}'}$. Lastly, we note that
    \[\frac{L_-(w)}{2\sin(w/2)}=\frac{1}{2}\sup_{(x, y)\in\mathcal{R}_w}(x^{2/3}+y^{2/3})^{3/2}.\]
    Indeed, this quantity decreases as $w$ increases, by viewing it as optimizing the same function over nested domains.
    We also deduce that $L_-(w)\ge \sin\theta_-$ for $w\in(0, \pi/2)$, allowing us to simplify \eqref{deducesigup}. Hence $\sigma_u(w)=\frac{L_-(w)}{2\sin(w/2)}$ and the assertion thus follows.
\end{proof}

\section{Upper bounds above the central width}\label{sec:upper-bound-large-width-case}

\subsection{Dual construction}

First notice that the construction in Section~\ref{sec:upper-bound-small-width-case} is not applicable for $w\in(\pi/2, \pi)$ anymore since $\mathbb{B}^n(w)$ is not even convex. This thus requires a new technique.
Let \(w\in(\pi/2,\pi)\) and put \({w}':=\pi-w\in(0,\pi/2)\). By the previous construction, there exists a spherical convex body \(M_{{w}'}\subset \mathbb S^n\) of constant width \(w'\). Define its {spherical dual} by
\[
M_{w'}^*:=\{y\in \mathbb S^n:\langle y,x\rangle\ge0 \text{ for all }x\in M_{w'}\}=\bigcap_{x\in M_{w'}}H_x.
\]
For more information related to this concept, see \cite[Lemma~4.1]{BW2016}.
By \cite{HW}, \(M_w:=M_{w'}^*\) is a spherical convex body of
constant width \(\pi-w'=w\). Thus the large-width case can be obtained from the small-width one. Again, from now on, $M$ stands for $M_w$.

\begin{lemma}\label{bounddual}
    Let $u\in\mathbb{S}^{n-1}$. Decompose $u=u^+-u^-$ with $\cos\varphi=\|u^+\|$ and $\sin\varphi=\|u^-\|$ where $0\le\varphi\le\pi/2$. Let $v=\exp_p(r(-u))=(\cos r, -(\sin r)u)\in M$ with $r\ge 0$. Denote $\theta_+:=\cos^{-1}\sqrt{\cos w'}$. If $u^+, u^-\ne 0$, then we have
    \[r\le \rho(\varphi):=\tan^{-1}\frac{\cot\theta_+}{\cos\varphi}.\]
\end{lemma}
\begin{proof}
First note that $r\le \pi/2$.
    Since $v\in M$, $\langle v, z\rangle\ge 0$ for all $z\in M_{w'}$. In particular, pick
    \[z=\exp_p(\theta_+ u_0)=(\cos\theta_+, (\sin\theta_+) u_0)\in L_1\subset M_{w'},\]
    where $u_0:=u^+/\|u^+\|$. This thus gives
\begin{align*}
    \langle v, z\rangle=\cos r\cos\theta_+-\sin r\sin\theta_+\cos\varphi\ge 0 \quad \Rightarrow\quad \tan r\le \frac{\cot\theta_+}{\cos\varphi}
\end{align*}
as desired. 
\end{proof}

Intuitively, \(\exp_p(r(-u))\) is the point reached by starting at \(p\) and traveling geodesic distance \(r\) in the direction \(-u\). Lemma~\ref{bounddual} shows that the positive and negative components of \(u\) restrict how far one may travel in that direction while remaining in \(M\) (or slightly out).

\begin{remark}\label{boundary}
    It is not hard to see that
    \begin{align*}
       M\cap T_1=\mathbb{B}^n\left(w-\frac{\pi}{2}+\theta_+\right)\cap T_1\quad\text{and}\quad
   M\cap T_2=\mathbb{B}^n\left(\frac{\pi}{2}-\theta_+\right)\cap T_2.\end{align*}
\end{remark}

\subsection{Effective radius computation}

This step is similar to the previous section. Consider a $(k, n-k)$ orthant $Q$. If $k=0$ or $k=n$, refer to Remark~\ref{boundary}.
Now fix $1\le k\le n-1$. After reordering coordinates, we can parametrize a cover $K_Q\supset M\cap (-Q)$, using Lemma~\ref{bounddual}, as
\begin{align*}
    \Psi_Q(r, \varphi, s, t):=(\cos r, -\sin r\cos\varphi \psi_1(s), \sin r\sin\varphi\psi_2(t)),
\end{align*}
with $0\le r\le \rho(\varphi)$ and $0\le\varphi\le\pi/2$. In this case, we find
\begin{align*}
    G(\Psi_Q)=\begin{pmatrix}
        B & O & O\\
        O & \sin^2r\cos^2\varphi (G(\psi_1)) & O\\
        O & O & \sin^2r\sin^2\varphi(G(\psi_2))
    \end{pmatrix},\quad B:=\begin{pmatrix}
        1 & 0\\ 0 & \sin^2r
    \end{pmatrix}.
\end{align*}
By computing $\sqrt{\det G(\Psi_Q)}$ directly, we have the volume setup
\begin{align*}
    \mu_n(K_Q)= \frac{\omega_k\omega_{n-k}}{2^n}\int_0^{\pi/2}\int_0^{\rho(\varphi)}\cos^{k-1}\varphi\sin^{n-k-1}\varphi \sin^{n-1}r\, dr\, d\varphi.
\end{align*}
Indeed, $\mu_n(M\cap(-Q))\le\mu_n(K_Q)$. Therefore if $\mathcal{K}:=\bigcup_QK_Q$, then  $\mu_n(M)\le \mu_n(\mathcal{K})$. Analogously, we find $\mu_n(\mathcal{K})=W_1+W_2$ where
\begin{align*}
    W_1=\frac{\omega_n}{2^n}\left(\int_0^{w-\frac{\pi}{2}+\theta_+}\sin^{n-1}x\, dx+\int_0^{\frac{\pi}{2}-\theta_+}\sin^{n-1}x\, dx\right)
\end{align*}
and
\begin{align*}
    W_2=\sum_{k=1}^{n-1}\binom{n}{k}\frac{\omega_k\omega_{n-k}}{2^n}\int_0^{\pi/2}\cos^{k-1}\varphi\sin^{n-k-1}\varphi \left(\int_0^{\rho(\varphi)}\sin^{n-1}rdr\right)\, d\varphi.
\end{align*}
For convenience, denote $L_+(w):=\sup_{\varphi\in[0, \pi/2]}h_w(\varphi)$ where
\[h_w(\varphi):=(\cos^{2/3}\varphi+\sin^{2/3}\varphi)^{3/2}\sin \rho(\varphi)=\frac{\cot\theta_+ (\cos^{2/3}\varphi+\sin^{2/3}\varphi)^{3/2}}{\sqrt{\cos^2\varphi+\cot^2\theta_+}}.\]
By the same asymptotic argument, we yield
\begin{align*}
    \left(\frac{W_1}{\mu_n(\mathbb{B}^n(w/2))}\right)^{1/n}\sim \frac{\sin(w-\frac{\pi}{2}+\theta_+)}{2\sin(w/2)}
\quad\text{and}\quad
    \left(\frac{W_2}{\mu_n(\mathbb{B}^n(w/2))}\right)^{1/n} &\lesssim \frac{L_+(w)}{2\sin(w/2)}.
\end{align*}
Thus we have
\[\sighigh(w)\le\frac{1}{2\sin(w/2)}\max\left\{\sin\left(w-\frac{\pi}{2}+\theta_+\right), L_+(w)\right\}=:\sigma_u(w).\]
Finally we solve for $L_+$.

\begin{proposition}
The value of $L_+(w)$ is attained at
\begin{align*}
    \varphi_*=\cos^{-1}\left(\left(\frac{\cos^3w+\sqrt{-\cos^3w}}{1+\cos^3w}\right)^{1/2}\right).
\end{align*}
\end{proposition}
\begin{proof}
It suffices to consider $h_w^2(\varphi)$. Substituting $x=\cos\varphi$ gives
\[\ell_w(x):=h_w^2( \cos^{-1}x)=\frac{-\cos w(x^{2/3}+(1-x^2)^{1/3})^3}{(1+\cos w)x^2-\cos w}\]
where $x\in[0, 1]$. Computing $\ell_w'(x)$ with some simplification, we need to solve
\begin{align*}
    \cos w+x^{4/3}(1-x^2)^{-2/3}=0\quad\Leftrightarrow\quad \frac{x^4}{(1-x^2)^2}=-\cos^3 w.
\end{align*}
Indeed, we can write this as a quadratic equation (by $y=x^2$) and solve it directly. By the first derivative test, one can check that this unique critical point is the maximum.
\end{proof}

\begin{remark}
   In fact, by direct substitution, we have 
    \[L_+(w)=1+\sqrt{-\cos w}\ge \sin\left(w-\frac{\pi}{2}+\theta_+\right)\quad\text{and}\quad\sigma_{u}(w)=\frac{1+\sqrt{-\cos w}}{\sqrt{2(1-\cos w)}}.\]
    Indeed, one can check that $\sigma_u$ is increasing on $(\pi/2, \pi)$.  Also, note the boundary cases: \[\lim_{w\to\frac{\pi}{2}^+}\sigma_u(w)=\frac{1}{\sqrt{2}}\quad\text{and}\quad\lim_{w\to \pi^-}\sigma_u(w)=1.\]
\end{remark}

We thus have the upper bound $\sigma_u$ proposed in Theorem \ref{thm:main-intro}.

\section{Schramm-type lower bounds}
\label{sec:lower}

\subsection{Ball-containment argument}

We adopt an argument similar to \cite{schr}.
Define the map $T : \mathbb{R}^{n+1}\to\mathbb{R}^{n+1}$ as the symmetry about the first coordinate axis: \[T(x_1,x_2,\dots,x_{n+1})=(x_1,-x_2,\dots,-x_{n+1}).\]

In addition, below $\mathrm{cl}(\cdot)$ and $\mathrm{int}(\cdot)$ stand for, respectively, the closure and the interior of a set on a sphere (with respect to the subspace topology). We first prove the following.

\begin{lemma}\label{shortest}
    Suppose $0<\alpha,\varepsilon<\pi/2$, $0<\gamma<\pi$ and $x\in\mathbb{S}^n$ is such that $d(x,p)=\varepsilon$ (then also $d(Tx,p)=\varepsilon$). Assume $\gamma\le\alpha+\varepsilon$ and $\cos\gamma\le\cos\alpha\cos\varepsilon$. Let
    \[
    A_1:=\mathbb{B}^n(\alpha)\cap\mathrm{cl}\left(\mathbb{S}^n\setminus \mathbb{B}^n(x,\gamma)\right)
    \quad\text{and}\quad
    A_2:=\mathbb{B}^n(\alpha)\cap\mathrm{cl}\left(
    \mathbb{S}^n\setminus \mathbb{B}^n(Tx,\gamma)
    \right).
    \]
    Then we have
    \[d(A_1, A_2):=\inf\{d(a_1, a_2) : a_1\in A_1, a_2\in A_2\}= 2\sin^{-1}\left(\frac{\cos\alpha\cos\varepsilon-\cos\gamma}{\sin\varepsilon}\right).\]
\end{lemma}
\begin{proof} After a rotation, assume \(x=\cos\varepsilon \, e_1+\sin\varepsilon\, e_2\), and \(Tx=\cos\varepsilon \, e_1-\sin\varepsilon \, e_2\). Put \[ q(\alpha, \varepsilon, \gamma):=\frac{\cos\alpha\cos\varepsilon-\cos\gamma}{\sin\varepsilon}\ge0. \] Take \(y=(a,b,u)\in A_1\) and \(z=(a', b', v)\in A_2\) with $a,a',b,b'\in\mathbb{R}$ and $u, v\in\mathbb{R}^{n-1}$. Since \(y\in\mathbb B^n(\alpha)\), we have \(a\ge\cos\alpha\), while \(y\notin\mathbb B^n(x,\gamma)\) gives \[a\cos\varepsilon+b\sin\varepsilon\le\cos\gamma.\] Hence $b\le -q$. Similarly, \(z\in A_2\) implies \( b'\ge q\). Therefore \[ \|y-z\|\ge |b'-b|\ge 2q. \] Since \(\|y-z\|=2\sin(d(y,z)/2)\), we obtain \(d(y,z)\ge 2\sin^{-1}q\). Taking the infimum over \(y\in A_1\) and \(z\in A_2\) gives the lower bound. Finally the equality is achieved by taking $y=(\cos\alpha, -q, u)$ and $z=(\cos\alpha, q, u)$ with $\|u\|^2=1-\cos^2\alpha-q^2\ge 0$.
\end{proof}

\begin{remark}
    When $\gamma>\alpha+\varepsilon$, we have $A_1, A_2=\emptyset$. If $\cos\gamma\ge \cos\alpha\cos\varepsilon$, then $A_1\cap A_2\ne\emptyset$, implying $d(A_1, A_2)=0$. 
\end{remark}

In the next step, for $A\subset\mathbb{S}^n$ and $\gamma>0$, denote
\[
A^\gamma:=\bigcap_{x\in A}\mathbb{B}^n(x,\gamma).
\]

\begin{proposition}\label{ballcontainedlem}
Let \(0<\alpha,\delta<\pi/2\) and \(0<\gamma<\pi\). Suppose \(A\subset\mathbb B^n(\alpha)\subset\mathbb S^n\) and \(\mathrm{diam}(A)\le\delta\). Put \[ \beta(\alpha, \gamma, \delta):= \cos^{-1}\left( \frac{\cos\alpha\cos\gamma+\sin\frac{\delta}{2} \sqrt{\cos^2\alpha-\cos^2\gamma+\sin^2\frac{\delta}{2}}} {\cos^2\alpha+\sin^2\frac{\delta}{2}} \right), \] whenever it is well-defined. Assume moreover that, for every \(0\le\varepsilon\le\beta\), either $\gamma>\alpha+\varepsilon$ or the hypotheses of Lemma~\ref{shortest} are satisfied. Then $A^\gamma\cup (TA)^\gamma\supset \mathrm{int}(\mathbb B^n(\beta))$, so \[ \mu_n(A^\gamma)\ge \frac{1}{2}\mu_n(\mathbb B^n(\beta)). \]
\end{proposition}
\begin{proof} 
    Assume to the contrary that there exists $x\in \mathrm{int}\mathbb{B}^n(\beta)$ with $x\not\in A^{\gamma}\cup (TA)^{\gamma}$. Write $\varepsilon:=d(x, p)<\beta$. If $\gamma>\alpha+\varepsilon$, then for all $a\in A$, by the triangle inequality,
     \[d(x, a)\le d(x, p)+d(p, a)\le \varepsilon+\alpha<\gamma,\]    and similarly $d(x, Ta)<\gamma$. This reaches a contradiction.

    Now suppose the hypotheses of Lemma~\ref{shortest} are met.
    Since $x\not\in A^{\gamma}$, there exists $y\in A$ such that $d(x, y)>\gamma$. Similarly, as $x\not\in (TA)^{\gamma}$ and $T^{-1}=T$, we have $Tx\not\in A^{\gamma}$ and so there exists $z\in A$ with $d(Tx, z)>\gamma$. These, together with $y, z\in\mathbb{B}^n(\alpha)$, imply $y\in A_1$ and $z\in A_2$ as defined in Lemma~\ref{shortest}. Thus we have
    \begin{align*}
        2\sin^{-1}\left(\frac{\cos\alpha\cos\varepsilon-\cos\gamma}{\sin\varepsilon}\right)\le d(y, z)\le\delta.
    \end{align*}
    By manipulation, we obtain a quadratic inequality
    \[\left(\cos^2\alpha+\sin^2\frac{\delta}{2}\right)\cos^2\varepsilon-2\cos\alpha\cos\gamma\cos\varepsilon+\left(\cos^2\gamma-\sin^2\frac{\delta}{2}\right)\le 0\]
    and solving this gives $\varepsilon\ge\beta$. This is a contradiction.
\end{proof}

\subsection{Proof of the lower bound}

We are ready to provide a lower bound $\sigma_{\ell}^{\text{Sch}}$ by using Proposition~\ref{ballcontainedlem}. Let $K\in\mathcal{C}_n(w)$. There are two cases.
\begin{itemize}
    \item {\it Case 1: $w\ge\pi/2$.} Take $A=K^*$. Consider $\gamma=\pi/2$ (so $A^{\gamma}=A^*=K$), $\delta=\pi-w$, and 
    $\alpha_n=\sin^{-1}\left(\sqrt{\frac{2 n}{n+1}}\cos\frac{w}{2}\right)$. This comes from Jung's theorem (see \cite{BVD}), applied to the set
$A=K^*$ of diameter $\pi-w$. We assume the ball is centered at $p$.
    Thus
    \begin{align}\label{rerr}
        \operatorname{rer}_w(K)\ge\sigma_n\ge2^{-1/n}\left(\frac{\mu_n(\mathbb{B}^n(\beta_n))}{\mu_n(\mathbb{B}^n(w/2))}\right)^{1/n},
    \end{align}
    where $\beta_n:=\beta(\alpha_n, \gamma, \delta)$.
   Denote $\alpha_{\infty}=\lim_{n\to\infty}\alpha_n$ and $\beta_{\infty}=\beta(\alpha_{\infty}, \gamma, \delta)$. Then taking the limit infimum on both sides of \eqref{rerr} gives
    \begin{align*}
        \siglow(w)\ge \frac{\sin \beta_{\infty}}{\sin(w/2)}=\frac{2\sqrt{-\cos w}}{1-\cos w}=:\sigma_{\ell}^{\text{Sch}}(w).
    \end{align*}

    \item {\it Case 2: $w<\pi/2$.} Take $A=K$. Consider $\gamma=w$ (so $A^{\gamma}=A=K$), $\delta=w$, and $\alpha_n=\sin^{-1}\left(\sqrt{\frac{2n}{n+1}}\sin\frac{w}{2}\right)$ (also from Jung's theorem). Then we may proceed the same argument.
\end{itemize}
\noindent
For completeness, we also need to justify that the hypotheses for Proposition~\ref{ballcontainedlem} hold. Indeed, since cosine is decreasing on $[0, \pi]$, it suffices to show that 
\begin{align*}
     \cos\gamma\le\cos\alpha_n\cos\beta_n \quad&\Leftrightarrow\quad\cos\gamma\sin\frac{\delta}{2}\le \cos\alpha_n\sqrt{\cos^2\alpha_n-\cos^2\gamma+\sin^2\frac{\delta}{2}}\\
     &\Leftrightarrow\quad 0\le (\cos^2\alpha_n-\cos^2\gamma)\left(\cos^2\alpha_n+\sin^2\frac{\delta}{2}\right)
     \\
     &\Leftrightarrow\quad \cos\gamma\le \cos\alpha_n.
\end{align*}
This can be easily verified to hold for $w\in(0, \pi)$ with our choice of variables.

\begin{remark}
   Note that $\sigma_{\ell}^{\text{Sch}}(w)$ is decreasing on $(0, \pi/2)$ and increasing on $(\pi/2, \pi)$. Additionally, observe the boundary cases:
   \begin{align*}
       \lim_{w\to 0^+}\sigma_{\ell}^{\text{Sch}}(w)=\sqrt{3}-1, \quad \lim_{w\to \pi/2}\sigma_{\ell}^{\text{Sch}}(w)=0,\quad\text{and}\quad\lim_{w\to \pi^-}\sigma_{\ell}^{\text{Sch}}(w)=1.
   \end{align*}
   When $w\to 0^+$, the result agrees with the Euclidean case established by Schramm~\cite{SchrammVol}.
\end{remark}

\begin{remark}
With the current method, $\siglow(2\pi/3)\ge\sqrt{8/9}$.
    An exponential improvement of the lower bound in this width case would lead to an exponential improvement of the upper bound on Borsuk's numbers using the method of~\cite{schr}, see also~\cite[Section~3.2]{Kalai-survey}.
\end{remark}


\section{Gaussian lower bounds near the central width}\label{sec:lower_new}

\subsection{The central-width case}

Let us first provide some setups. For a measurable set $A\subset \mathbb{S}^n$, define a normalized measure  \[\nu_n(A):=\frac{\mu_n(A)}{\mu_n(\mathbb{S}^n)}.\]
A  (convex) cone is $C$ a subset of $\mathbb{R}^n$ which is closed under linear combinations with non--negative coefficients. Define its dual by
\[C^*:=\{y\in\mathbb{R}^{n} : \langle x, y\rangle\ge 0 \text{ for every }x\in C\}.\]
    Furthermore, define $\mathcal{CO}_n$ as the collection of convex cone $C\subset\mathbb{R}^n$ such that $C^*\subset C$.

\begin{proposition}\label{proponpi/2}
For any $n\ge 2$, define $T_n:=\inf_{K\in\mathcal{C}_n(\pi/2)}\nu_n(K)$. Then we have
\begin{align}\label{expo}
    \liminf_{n\to\infty}T_n^{1/n}\ge\xi_*,
\end{align}
where $\xi_*= 0.43174\ldots$ is the unique root of the quartic equation
    \[\xi^4+\xi^2-2\sqrt{2}\xi+1=0\]
    in the interval $\left(\sqrt{2}-1,\frac{1}{2}\right)$.
Consequently, we have $\siglow(\pi/2)\ge\sqrt{2}\xi_*= 0.61057\ldots$.
\end{proposition}

\begin{remark}
    For the case of self-polar cones, the asymptotic estimate of Proposition~\ref{proponpi/2} was announced without a proof by Nazarov in \cite{NazarovMO}.
\end{remark}

The proof of Proposition \ref{proponpi/2} will consist of several ingredients, which we shall prove first.
Let $C_K:=\{rx : r\ge 0, x\in K\}\subset\mathbb{R}^{n+1}$ be the cone generated by $K\subset\mathbb{S}^n$. 
Note that a spherical convex body of width $\pi/2$ is self-polar. Hence $C_K^*=C_K=C_{K^*}$.

Next, let $\gamma_n$ denote the standard Gaussian measure on $\mathbb{R}^n$. For more information related to this, we refer the readers to \cite{Stroock}. By Fubini--Tonelli's theorem, we have
\begin{align*}
    \gamma_{n+1}(C_K)&=(2\pi)^{-\frac{n+1}{2}}\int_{C_K}e^{-\frac{\|x\|^2}{2}}\,dx=(2\pi)^{-\frac{n+1}{2}}\int_{K}\int_0^{\infty}e^{-\frac{r^2}{2}}r^n\,dr\,d\mu_n(\theta)\\
    &=\left(\int_K\,d\mu_n(\theta)\right)\left((2\pi)^{-\frac{n+1}{2}}\int_0^{\infty}e^{-\frac{r^2}{2}}r^n\,dr\right)=\frac{\mu_n(K)}{\mu_n(\mathbb{S}^n)}=\nu_n(K).
\end{align*}
It therefore remains to lower-bound the Gaussian measure of a self-dual cone. We first prove an explicit exponential lower bound valid in all dimensions.

\begin{lemma}\label{root2-1}
    For all $n\ge 2$ and $C\in\mathcal{CO}_n$, we have $\gamma_n(C)\ge(1+\sqrt{2})^{-n}$.
\end{lemma}
\begin{proof}
    Fix one such cone $C$, and write $\kappa:=\gamma_n(C)$. Let
    \[\phi_n(x):=(2\pi)^{-\frac{n}{2}}e^{-\frac{\|x\|^2}{2}}\]
    be the Gaussian density function. Put $h(x):=\mathbf{1}_C(x)\phi_n(x)$, where $\mathbf{1}_C$ stands for the indicator function of the set $C$. Clearly we have
    $\int_{\mathbb{R}^n}h(x)dx=\kappa$.
    Define the autocorrelation
    \[F(z):=\int_{\mathbb{R}^n}h(x)h(x-z)\,dx=\int_{C\cap (z+C)}\phi_n(x)\phi_n(x-z)\,dx,\]
    where $z+C:=\{z+c : c\in C\}$.
    Observe that, by Fubini--Tonelli's theorem,
    \begin{align}\label{identityp}
    \begin{split}
        \int_{\mathbb{R}^n}F(z)\, dz &=\int_{\mathbb{R}^n}\int_{\mathbb{R}^n}h(x)h(x-z)\,dx\,dz=\int_{\mathbb{R}^n}h(x)\left(\int_{\mathbb{R}^n}h(x-z)\,dz\right)\,dx\\
        &=\left(\int_{\mathbb{R}^n}h(x)\,dx\right)^2=\kappa^2.
    \end{split}
    \end{align}
    Next we invoke the Moreau decomposition \cite{Moreau}:  for every $z\in\mathbb{R}^n$, there are $u\in C$ and $v\in C^*\subset C$ so that $z=u-v$ and $\langle u, v\rangle=0$. Note that, because the cone is closed under addition, we have $u+C\subset C\cap (z+C)$. Put $s=\|z\|$ and $y=u+v$. By orthogonality, 
    $\|y\|=\|u+v\|=\|u-v\|=\|z\|=s$.
    So for $w\in C$, both $u+w$ and $v+w=(u+w)-z$ belong to $C$. Hence, by some expansion, we obtain
    \begin{align}\label{keyineq}
        \begin{split}
            F(z) &\ge\int_{u+C}\phi_n(x)\phi_n(x-z)\, dx
            =\int_C\phi_n(u+w)\phi_n(v+w)\, dw
            \\&\ge
            (2\pi)^{-n}\int_C\exp\left(-\frac{\|u+w\|^2+\|v+w\|^2}{2}\right)\, dw\\
            &=(2\pi)^{-n}e^{-s^2/2}\int_Ce^{-\|w\|^2-\langle y, w\rangle}\, dw.
        \end{split}
    \end{align}
    Next, by Cauchy--Schwarz, $\langle y, w\rangle\le s\|w\|$. As the remaining integrand in \eqref{keyineq} is radial, we may perform integration in polar coordinates, by writing $w=r\theta$ where $r=\|w\|$:
    \begin{align}\label{newineq}
    \begin{split}
        F(z) &\ge (2\pi)^{-n}\mu_{n-1}(C\cap \mathbb{S}^{n-1})e^{-\frac{s^2}{2}}\int_0^{\infty}r^{n-1}e^{-r^2-sr}\,dr
        \\&= \kappa(2\pi)^{-n}\mu_{n-1}(\mathbb{S}^{n-1})e^{-\frac{s^2}{2}}\int_0^{\infty}r^{n-1}e^{-r^2-sr}\,dr.
    \end{split}
    \end{align}
    Integrating \eqref{newineq} over $z\in\mathbb{R}^n$ (recall: $s=\|z\|$), using \eqref{identityp}, and cancelling $\kappa$, we yield
    \begin{align}
        \kappa\ge L_n(1),
    \end{align}
    where, for $0<\vartheta\le 1$, define
    {\begin{align}\label{Lnalpha}
    \begin{split}
        L_n(\vartheta) &:=(2\pi)^{-n}\omega_{n}^2\int_0^{\infty}\int_0^{\infty}s^{n-1}r^{n-1}e^{-\frac{s^2}{2}-r^2-\vartheta sr}\,dr\,ds
        \\
        &=\frac{\Gamma(\frac{n+1}{2})}{\sqrt{\pi}\Gamma(\frac{n}{2})}\int_0^{\infty}\left(\frac{t}{t^2+\vartheta t+1/2}\right)^n\,\frac{dt}{t}\\
        &=\frac{\Gamma(\frac{n+1}{2})}{\sqrt{\pi}\Gamma(\frac{n}{2})}\int_{\mathbb{R}}(\vartheta+\sqrt{2}\cosh x)^{-n}\,dx.
    \end{split}
    \end{align}}
    Note that the second equality follows by putting $t=r/s$ and integrating first in $s$, while third follows from substituting $t=e^x/\sqrt{2}$.

    Now, we find an exponential lower bound for $L_n(1)$. 
    Put $\mathfrak{b}:=\sqrt{2}-1$, $\mathfrak{c}:=2-\sqrt{2}$, and
    \[f(x):=\log\frac{1+\sqrt{2}\cosh x}{1+\sqrt{2}}\]
    for $x\in\mathbb{R}$. Observe that $f(0)=f'(0)=0$ and 
    \begin{align*}
        f''(x)=\frac{2+\sqrt{2}\cosh x}{(1+\sqrt{2}\cosh x)^2}\le \mathfrak{c}.
    \end{align*}
    Thus using Taylor's theorem with integral remainder, we obtain
    \[f(x)=f(0)+xf'(0)+\int_0^x(x-t)f''(t)\, dt\le\frac{\mathfrak{c}x^2}{2}\]
    for $x\ge 0$. As $f$ is even, the same bounds also holds for $x\le 0$.
    It then follows that
    \begin{align}\label{bound1}
        \int_{\mathbb{R}}(1+\sqrt{2}\cosh x)^{-n}\,dx=\int_{\mathbb{R}}\mathfrak{b}^ne^{-nf(x)}\,dx\ge \mathfrak{b}^n\sqrt{\frac{2\pi}{\mathfrak{c}n}}
    \end{align}
    via standard Gaussian integral formula. Also recall Kershaw's inequality \cite{Kershaw}:
    \begin{align}\label{bound2}
        \frac{\Gamma(\frac{n+1}{2})}{\Gamma(\frac{n}{2})}\ge\sqrt{\frac{n}{2}-\frac{1}{4}}
    \end{align}
    for $n\ge 2$. Combining \eqref{bound1}, \eqref{bound2} gives $L_n(1)\ge \mathfrak{b}^n$ as desired.
\end{proof}

Lemma~\ref{root2-1} provides an initial exponential bound. In fact, this type of argument can provide a better estimate. Recall that the main task is to bound $F$ from below by integrating
$
\varrho_z(w):=e^{-\|w\|^2-\langle y,w\rangle}$
over the cone $C$. Rather than integrating over all of $C$, we restrict to a suitable subset on which $\varrho_z(w)$ admits a sharper lower bound. Specifically, instead of plain Cauchy--Schwarz bound of $\langle y,w\rangle$, an improved estimate uses information about the angle between $y$ and $w$. Asymptotically, the resulting gain outweighs the loss from restricting the domain.

\begin{lemma}\label{prelim}
    For $n\ge 2$, define
    $M_n:=\inf\{\gamma_n(C) : C\in\mathcal{CO}_n\}$.
    Then \[\liminf_{n\to\infty}M_n^{1/n}\ge \xi_*.\]
\end{lemma}
\begin{proof}
    For fixed $0<\vartheta<1$, let $\mathcal{U}_n(\eta, \vartheta):=\{\theta\in\mathbb{S}^{n-1} : \langle \eta, \theta\rangle>\vartheta\}$ and
    \begin{align*}
        q_n(\vartheta):=\frac{\mu_{n-1}(\{\theta\in\mathbb{S}^{n-1} : \langle \eta, \theta\rangle>\vartheta\})}{\mu_{n-1}(\mathbb{S}^{n-1})}
    \end{align*}
    where $\eta\in\mathbb{S}^{n-1}$ is arbitrary. For $z\ne 0$, take $\eta=y/s$. Define
    \[\mathcal{D}_n(\eta, \vartheta):=\{\theta\in C\cap\mathbb{S}^{n-1} : \langle \eta, \theta\rangle\le\vartheta\}.\]
    Notice that $(C\cap\mathbb{S}^{n-1})\setminus\mathcal{U}_n(\eta, \vartheta)=\mathcal{D}_n(\eta, \vartheta)$.  By the property of measure, 
    \begin{align}\label{comparison}
        \nu_{n-1}(\mathcal{D}_n(\eta, \vartheta))\ge\nu_{n-1}(C\cap \mathbb{S}^{n-1})-\nu_{n-1}(\mathcal{U}_n(\eta, \vartheta))=\kappa-q_n(\vartheta).
    \end{align}
    Next take a good direction $\theta$ so that $\langle \eta, \theta\rangle\le\vartheta$. Write $w=r\theta$ and $r=\|w\|$ as usual. Since $y=s\eta$, we have
    \[\langle y, w\rangle=\langle s\eta, r\theta\rangle=sr\langle \eta, \theta\rangle\le \vartheta sr.\]
    So, by instead integrating along these good directions, we have
    \begin{align}\label{better}
    \begin{split}
        F(z) &\ge (2\pi)^{-n}e^{-\frac{s^2}{2}}\int_{C_{\mathcal{D}_n(\eta, \vartheta)}}e^{-\|w\|^2-\langle y, w\rangle}\,dw
        \\&\ge(2\pi)^{-n}e^{-\frac{s^2}{2}}\mu_{n-1}(\mathcal{D}_n(\eta, \vartheta))\int_0^{\infty}r^{n-1}e^{-r^2-\vartheta s r}\,dr.
    \end{split}
    \end{align}
    Now, integrating \eqref{better} over $z\in\mathbb{R}^n$ and using \eqref{identityp}, \eqref{comparison}, we find
    \begin{align}\label{kappasquareineq}
        \kappa^2\ge (\kappa-q_n(\vartheta))L_n(\vartheta).
    \end{align}
    Let us keep this in mind. Next we see that
    \begin{align}\label{anotherclaim}
        (q_n(\vartheta))^{1/n}\sim\left(\frac{\mu_{n-1}(\mathbb{B}^{n-1}(\cos^{-1}\vartheta))}{\mu_{n-1}(\mathbb{S}^{n-1})}\right)^{1/n}\sim \sin(\cos^{-1}\vartheta)=\sqrt{1-\vartheta^2}.
    \end{align}
    Moreover, by the $L^p$-to-$L^{\infty}$-limit,
    \begin{align}\label{claim}
        (L_n(\vartheta))^{1/n}\sim \sup_{x\in\mathbb{R}}\frac{1}{\vartheta+\sqrt{2}\cosh x}
      =\frac1{\sqrt2+\vartheta}.
    \end{align}
    
    In the final step, set $\xi:=\liminf_{n\to\infty}M_n^{1/n}$. Lemma \ref{root2-1} gives $\xi\ge \sqrt{2}-1$. Also, since $\mathcal{O}_{n-1}$ is self-dual and has Gaussian measure $2^{-n}$, $\xi\le 1/2$. Choose
    $\vartheta\in(0, 1)$ so that $\sqrt{1-\vartheta^2}<\xi$. By \eqref{anotherclaim}, \eqref{claim}, the definition of $\xi$, and the condition of $\vartheta$, there exists a constant $\lambda\in(0, 1)$ so that
    \begin{align}\label{expodecay}
        \frac{q_n(\vartheta)}{M_n}=O(\lambda^n).
    \end{align}
    Dividing \eqref{kappasquareineq} by $\kappa$, using $\kappa\ge M_n$, and taking the infimum over all $C\in\mathcal{CO}_n$, we have
    \[M_n\ge\left(1-\frac{q_n(\vartheta)}{M_n}\right)L_n(\vartheta).\]
    Then taking the $n$-rooth, letting $n\to\infty$, and using \eqref{expodecay}, we have $\xi\ge(\vartheta+\sqrt{2})^{-1}$. Letting $\vartheta\downarrow\sqrt{1-\xi^2}$, we obtain the inequality
    \begin{align*}
        \xi\ge\frac{1}{\sqrt{2}+\sqrt{1-\xi^2}}\quad\Rightarrow\quad  \xi^4+\xi^2-2\sqrt{2}\xi+1\le 0.
    \end{align*}
    By the sign test, we find that $\xi\ge \xi_*$ as described in the assertion.
\end{proof}

\begin{remark}
    While Lemma~\ref{prelim} gives a better lower bound than Lemma~\ref{root2-1}, we use Lemma~\ref{root2-1} in the proof when  
    $\vartheta\in(0, 1)$ is selected, which is possible due to $\xi>0$. Besides that, Lemma~\ref{root2-1} provides an explicit estimate for all dimensions, while the result of Lemma~\ref{prelim} is asymptotic.
\end{remark}

\begin{proof}[Proof of Proposition~\ref{proponpi/2}]
    Put $n=m+1$ and let $C_K\subset\mathbb{R}^n$ be the cone over $K$. For $K\in\mathcal{C}_m(\pi/2)$, we have $C_K=C_K^*$ and $\nu_m(K)=\gamma_n(C_K)$. \eqref{expo} then follows from Lemma~\ref{prelim}. This, together with  $(\nu_m(\mathbb{B}^m(\pi/4)))^{1/m}\sim\sin(\pi/4)$,
    we obtain the lower bound for $\siglow(\pi/2)$.
\end{proof}

In subsequent subsections, we will adapt this argument to also find a lower bound for widths near the central $\pi/2$.

\subsection{Neighboring widths (below the central)}

The difficulty is that the cone generated by a body of constant width slightly less than $\pi/2$ is no longer self-dual. Nevertheless, its dual cone remains quantitatively close to the original cone, allowing us to modify the Moreau decomposition argument.

For a spherical convex body $L\subset\mathbb{S}^n$ and $t\ge 0$, define its outer parallel body by
\begin{align*}
    L_t:=\{y\in\mathbb{S}^n : d(y, L)\le t\}.
\end{align*}
Also, define
    $T_n(w):=\inf_{K\in\mathcal{C}_n(w)}\nu_n(K)$ for $w\in(0, \pi)$.
We will also use several notations from the previous subsection. Let us first prove important properties of parallel bodies.

\begin{lemma}\label{parallelidentity}
    Let $E\subset\mathbb{S}^n$ be non-empty and closed. Then the following hold.
    \begin{enumerate}
        \item[(a)] If $E\in\mathcal{C}_n(\delta)$, where $\delta\in(0, \pi/2)$, and $0\le s\le\pi/2-\delta$, then $E^{\delta+s}=E_s$.
        \item[(b)] If $0\le s<\pi/2$, then $(E_s)^*=E^{\pi/2-s}$.
        \item[(c)] If $s, u\ge 0$ and $s+u<\pi$, then $(E_s)_u=E_{s+u}$.
    \end{enumerate}
\end{lemma}
\begin{proof}
    (a) First let $x\in E_s$. Then there exists $y\in E$ so that $d(x, y)\le s$. Using the triangle inequality, for any $z\in E$, we have $d(x, z)\le d(x, y)+d(y, z)\le s+\delta$. Conversely, if $x\in E^{\delta+s}$, then $\max_{y\in E}d(x, y)\le\delta+s$. If $x\in E$, then it is trivial. If not, since $E$ is complete,  
    \[\max_{y\in E}d(x, y)=\delta+d(x, E),\]
    provided that the right-hand side is $<\pi$. Combining these gives $d(x, E)\le s$.
    \\
    (b) If $x\in E^{\pi/2-s}$ and $y\in E_s$, then choose $z\in E$ with $d(y, z)\le s$. The triangle inequality gives
    $d(x, y)\le d(x, z)+d(y, z)\le\pi/2$, implying $x\in (E_s)^*$. Conversely, if $x\not\in E^{\pi/2-s}$, choose $y\in E$ with $d(x, y)>\pi/2-s$. Extend the geodesic path from $x$ through $y$ past by distance $\le s$. The resulting point $z$ belongs to $E_s$ and satisfies $d(x, z)>\pi/2$. Hence $x\not\in (E_s)^*$.
    \\
    (c) The inclusion $(E_s)_u\subset E_{s+u}$ is easy. For the reverse inclusion, take $x\in E_{s+u}$ and choose $y\in E$ with $d(x, y)\le s+u$, and choose on $[y, x]$ a point $z$ with $d(y, z)\le s$ and $d(x, z)\le u$. Then $z\in E_s$ and $x\in(E_s)_u$.
\end{proof}

  Now we may state the result.


\begin{proposition}\label{finally}
    For $0<w\le \pi/2$, define $R(w):=1+2\cot w$ and $D(w):=\sqrt{1+(R(w))^2}$. Let $\xi(w)$ be the smallest positive root of the quartic equation
    \begin{align*}
        (R(w))^2\xi^4+\xi^2-2(D(w))\xi+1=0.
    \end{align*}
    Then we have $\chi(w):=\liminf_{n\to\infty}(T_n(w))^{1/n}\ge\xi(w)$. Consequently,
    \[\siglow(w)\ge\frac{\xi(w)}{\sin(w/2)}=:\sigma_{\ell}^{\text{G}}(w).\]
\end{proposition}

\begin{remark}
    Notice that $\xi(\pi/2)=\xi_*$ since $R(\pi/2)=1$ and $D(\pi/2)=\sqrt{2}$.
\end{remark}

We first establish a quantitative replacement for self-duality.

\begin{lemma}\label{qualitativedual}
    Let $0<w\le\pi/2$ and $K\in\mathcal{C}_n(w)$. For every $v\in (C_K)^*$, there exists $a\in C_K$ such that $v+a\in C_K$ and 
    $\|a\|\le\|v\|\cot w$.
\end{lemma}
\begin{proof}
     Put $\mathfrak{q}_0=v$ and apply the Moreau decomposition recursively: for every $j\ge 0$,
    \begin{align}\label{moreaiterate}
        \mathfrak{q}_j=x_j-\mathfrak{q}_{j+1},\quad x_j\in C_K,\quad \mathfrak{q}_{j+1}\in C_K^*, \quad \langle x_j, \mathfrak{q}_{j+1}\rangle=0.
    \end{align}
    Note that Lemma \ref{parallelidentity} gives  \begin{align}\label{longchain}       K^*=K^{\pi/2}=K^{w+\delta}=K_{\delta},  \end{align}
    where $\delta=\pi/2-w$. Write $r_j=\|\mathfrak{q}_j\|$. Note that $x_j$ is the Euclidean projection of $\mathfrak{q}_j$ onto $C_K$ and $\mathfrak{q}_j/r_j\in K^*$. By \eqref{longchain}, we have
    \begin{align}\label{info}
       r_{j+1}=d_E(\mathfrak{q}_j, C_K)= r_jd_E(\mathfrak{q}_j/r_j, C_K)\le r_j \sin(d_E(\mathfrak{q}_j/ r_j, K))\le r_j\sin\delta,
    \end{align}
    where $d_E$ denotes the usual Euclidean $L^2$--distance.
    Also, by orthogonality,
    \begin{align}\label{info2}
         \|x_j\|^2=\|x_j-\mathfrak{q}_{j+1}\|^2-r_{j+1}^2={r_j^2-r_{j+1}^2}.
    \end{align}
    In particular, $r_j\le (\sin\delta)^jr_0$, so the following two series converge absolutely in $C_K$:
    \begin{align*}
        S_o:=\sum_{j\ge 0}x_{2j+1}\quad\text{and}\quad S_e:=\sum_{j\ge 0}x_{2j}.
    \end{align*}
    Iterating \eqref{moreaiterate} gives $\mathfrak{q}_0=x_0-x_1+x_2-\ldots+x_{2m}-\mathfrak{q}_{2m+1}$ and thus, by letting $m\to\infty$, we yield $v+S_o=S_e\in C_K$. We will take $a=S_o$. Next we estimate $\|S_o\|$. For $j, m\ge 0$, put
    \begin{align*}
        S_{j, m}:=\sum_{\ell=0}^m\|x_{j+2\ell}\|.
    \end{align*}
    
    We claim that $S_{j, m}\le r_j/\sin w$. This is done via induction on $m$. For the basic step, when $m=0$, using $\|x_j\|\le r_j$ gives $S_{j, 0}=\|x_j\|\le r_j/\sin w$. Now suppose that the assertion is known for $m-1$. Then $S_{j, m}=\|x_j\|+S_{j+2, m-1}$ and hence
    \[S_{j+2, m-1}\le\frac{r_{j+2}}{\sin w}.\]
    Using \eqref{info} and \eqref{info2}, we therefore obtain
    \begin{align}\label{boundSjm}
        S_{j, m}\le\sqrt{r_j^2-r_{j+1}^2}+\frac{r_{j+2}}{\sin w}\le\sqrt{r_j^2-r_{j+1}^2}+\frac{r_{j+1}\sin\delta}{\sin w}.
    \end{align}
    Indeed, if $r_j=0$, then all terms in \eqref{boundSjm} vanish and we are done. Now suppose $r_j>0$ and put $t:=r_{j+1}/r_j$. We can rewrite \eqref{boundSjm} as
    \begin{align}\label{boundSjm_new}
        \frac{S_{j, m}}{r_j}\le \sqrt{1-t^2}+\frac{t\sin\delta}{\sin w}.
    \end{align} It is thus enough to show that the right-hand side of \eqref{boundSjm_new} is bounded above by $1/\sin w$. This easily follows from Cauchy-Schwarz:
    \[\sin w\sqrt{1-t^2}+t\sin\delta\le \sqrt{\sin^2w+\sin^2\delta}\sqrt{1-t^2+t^2}=1.\]
    This completes the inductive step.
    Finally, the claim gives
    \begin{align*}
        \|S_o\|\le\sum_{j=0}^{\infty}\|x_{2j+1}\|\le\frac{r_1}{\sin w}\le\frac{r_0\sin\delta}{\sin w}=\|v\|\cot w.
    \end{align*}
    The proof is complete.
\end{proof}


\begin{proof}[Proof of Proposition~\ref{finally}]
    Fix $K\in\mathcal{C}_m(w)$ and put $n=m+1$. Define the autocorrelation $F$ as in Lemma~\ref{root2-1}. For $z\in\mathbb{R}^n$, the Moreau decomposition gives $z=u-v$, where $u\in C_K$, $v\in C_K^*$, and $\langle u, v\rangle=0$. By Lemma~\ref{qualitativedual}, there exists $a\in C_K$ such that $v+a\in C_K$ and
    \begin{align}\label{idd}
        \|a\|\le\cot w\|v\|.
    \end{align}
    Put $s=\|z\|$. Observe that $u+a, v+a\in C_K$. Moreover, \eqref{idd} gives \begin{align}\label{ineq1}
        \|u+v+2a\|\le \|u-v\|+2\|a\|\le R(w)s,
    \end{align}
    and by the parallelogram identity,
    \begin{align*}
        \|u+a\|^2+\|v+a\|^2=\frac{\|u-v\|^2+\|u+v+2a\|^2}{2}\le\frac{(D(w)s)^2}{2}.
    \end{align*}
    For every $q\in C_K$, we have  $u+a+q, v+a+q\in C_K$. Note the square expansion \[\|u+a+q\|^2+\|v+a+q\|^2=\|u+a\|^2+\|v+a\|^2+2\|q\|^2+2\langle u+v+2a, q\rangle.\] Hence, analogous to \eqref{keyineq},
    \begin{align*}
        F(z)&\ge\int_{u+a+C_K}\phi_n(x)\phi_n(x-z)dx
        \ge \frac{e^{-(D(w))^2s^2/4}}{(2\pi)^n}\int_{C_K}e^{-\|q\|^2-\langle u+v+2a, q\rangle}\,dq.
    \end{align*}
    
    We now repeat the good-direction argument from Lemma~\ref{prelim}. For fixed $0<\vartheta<1$ and $u+v+2a\ne 0$, restrict the $q$-integration to directions satisfying
    \[\left\langle \frac{u+v+2a}{\|u+v+2a\|}, \frac{q}{\|q\|}\right\rangle\le \vartheta.\]
    By \eqref{ineq1}, we have $\langle u+v+2a, q\rangle\le\vartheta R(w)s\|q\|$. The same polar-coordinate calculation as before therefore yields
    \begin{align}\label{modified}
        \kappa^2\ge(\kappa-q_n(\vartheta))L_{n, w}(\vartheta),
    \end{align}
    where
    \begin{align}\label{Lnw}
    \begin{split}
        L_{n, w}(\vartheta) &=\frac{\Gamma(\frac{n+1}{2})}{\sqrt{\pi}\Gamma(\frac{n}{2})}\int_0^{\infty}\left(\frac{t}{t^2+\vartheta R(w)t+(D(w))^2/4}\right)^n\,\frac{dt}{t}\\
         &=\frac{\Gamma(\frac{n+1}{2})}{\sqrt{\pi}\Gamma(\frac{n}{2})}\int_{\mathbb{R}}\left(\vartheta R(w)+D(w)\cosh x\right)^{-n}\,dx.
        \end{split}
    \end{align}
    Note that the replacements from \eqref{Lnalpha} are ${s^2}/{2}\mapsto {(D(w))^2s^2}/{4}$ and $\vartheta sr\mapsto \vartheta R(w)sr$.
    Analogously, we also have
    \begin{align}\label{asymptoticqL}
        (q_n(\vartheta))^{1/n}\sim\sqrt{1-\vartheta^2}\quad\text{and}\quad (L_{n, w}(\vartheta))^{1/n}\sim\frac{1}{D(w)+\vartheta R(w)}.
    \end{align}

    Taking $\vartheta=1$ in \eqref{modified} gives the initial bound $\chi(w)\ge\chi_0$ where $\chi_0:=(D(w)+R(w))^{-1}$.
    We may now repeat the bootstrap from Lemma~\ref{prelim}. If $\chi(w)\ge\chi_j>0$, then choose $\vartheta>\sqrt{1-\chi_j^2}$. So $q_n(\vartheta)/T_{n-1}(w)\to 0$ exponentially, and \eqref{modified}, \eqref{Lnw}, \eqref{asymptoticqL} gives
    \begin{align*}
        \chi(w)\ge\frac{1}{D(w)+\vartheta R(w)}.
    \end{align*}
    Letting $\vartheta\downarrow \sqrt{1-\chi_j^2}$, we obtain $\chi(w)\ge \chi_{j+1}$ where
    \begin{align*}
         \chi_{j+1}:=\frac{1}{D(w)+R(w)\sqrt{1-\chi_j^2}}.
    \end{align*}
    Using induction, one can show that the sequence $\{\chi_j\}$ increases to the smallest positive fixed point of this map, namely $\xi(w)$. By further algebraic manipulations,  we are done.
\end{proof}



\subsection{Neighboring widths (above the central)} Actually,  Lemma~\ref{prelim} is also applicable for $w>\pi/2$ with an immediate extension. However, we can improve our bound as shown below. The strategy is to pass to the polar body $K^*$, enlarge it by a suitable parallel-body construction to obtain a body $M\in\mathcal{C}_n(\pi/2)$, and write $K=M_t$. The spherical isoperimetric inequality then yields the desired lower bound for $\nu_n(K)$.

\begin{proposition}\label{lowerslightlymorethanpi/2}
    Let $n\ge 2$ and $w\in(\pi/2, \pi)$. Set $t:=(w-\pi/2)/{2}$. Suppose $a_n\in(0, 1)$ has the property that
    $\nu_n(L)\ge a_n$
    for every $L\in\mathcal{C}_n(\pi/2)$. Let $\theta_n\in(0, \pi)$ be determined by $\nu_n(\mathbb{B}^n(\theta_n))=a_n$.
    Then every $K\in\mathcal{C}_n(w)$ satisfies
    \begin{align}\label{eq35}
        \nu_n(K)\ge\nu_n(\mathbb{B}^n(\min\{\theta_n+t, \pi\})).
    \end{align}
    In particular, if $\liminf_{n\to\infty}a_{n}^{1/n}\ge \mathfrak{a}$ for some $\mathfrak{a}\in(0, 1)$, then \[\liminf_{n\to\infty}(T_n(w))^{1/n}\ge \Omega_t(\mathfrak{a}):=\begin{cases}
        \sin(\sin^{-1}\mathfrak{a}+t) &\text{if } \sin^{-1}\mathfrak{a}+t<\pi/2,\\
        1 &\text{if } \sin^{-1}\mathfrak{a}+t\ge\pi/2.
    \end{cases}\]
\end{proposition}

\begin{remark}
    In fact, since the spherical orthant has normalized measure $2^{-(n+1)}$, any admissible central-width lower bound necessarily satisfies $\mathfrak{a}\le 1/2$. Consequently, the first case in the piecewise definition of $\Omega_t$ always applies.
\end{remark}

 We can then utilize Proposition \ref{lowerslightlymorethanpi/2} to our interest in spherical bodies of constant width, with $a_n:=T_n$, $\liminf_{n\to\infty}(a_n)^{1/n}\ge \xi_*$, as $\xi_*<1/2$. This gives
 \begin{align*}
     \siglow(w)&\ge \frac{1}{\sin(w/2)}\sin\left(\sin^{-1}\xi_*+\frac{w-\pi/2}{2}\right)
     \\ &= \frac{1}{\sin(w/2)}\left(\xi_*\cos\frac{w-\pi/2}{2}+\sqrt{1-\xi_*^2}\sin\frac{w-\pi/2}{2}\right)=:\sigma_{\ell}^{\text{G}}(w).
 \end{align*}

 \begin{proof}[Proof of Proposition \ref{lowerslightlymorethanpi/2}]
     Fix $K\in\mathcal{C}_n(w)$ and put $A:=K^*$, $\delta:=\pi-w$. Set $M:=A_t$. We claim that $M\in\mathcal{C}_n(\pi/2)$. Indeed, since $\delta+2t=\pi/2$, Lemma \ref{parallelidentity} gives
     \begin{align*}
         M^*=(A_t)^*=A^{\pi/2-t}=A^{\delta+t}=A_t=M,
     \end{align*}
     implying $M$ is self-dual.
     As $\mathrm{diam}(M)=\pi/2$, it is a complete spherical body, implying the claim. In addition, Lemma \ref{parallelidentity} also gives the exact reconstruction of $K$ from $M$:
     \begin{align}\label{KM}
         K=A^*=A^{\pi/2}=A_{\pi/2-\delta}=A_{2t}=(A_t)_t=M_t.
     \end{align}

     Next, by the hypothesis, $\nu_n(M)\ge a_n$. Let $r_n$ be chosen so that $\nu_n(M)=\nu_n(\mathbb{B}^n(r_n))$. Then $r_n\ge\theta_n$. Using \eqref{KM} and \cite[Section~1.1, equation~(1.1)]{Ledoux1999}, we have
     \begin{align*}
         \nu_n(K)=\nu_n(M_t)\ge\nu_n((\mathbb{B}^n(r_n))_t)\ge\nu_n(\mathbb{B}^n(\min\{\theta_n+t, \pi\})),
     \end{align*}
     which is \eqref{eq35}.
     The next result then follows from the fact that $(\nu_n(\mathbb{B}^n(r)))^{1/n}\sim\sin r$. We have $\theta_n=\sin^{-1}\mathfrak{a}+o(1)$ whenever $\sin^{-1}\mathfrak{a}+t<\pi/2$. If $\sin^{-1}\mathfrak{a}+t\ge\pi/2$, the cap in \eqref{eq35} has measure bounded below by $1/2+o(1)$ whose $n$th root tends to $1$.
 \end{proof}

\begin{remark}
     Numerically, $\sigma_{\ell}^{\text{G}}$ beats $\sigma_{\ell}^{\text{Sch}}$ when
     \[0.44494\ldots\pi\le w\le  0.54424\ldots\pi.\]
 \end{remark}

In the end, combining the preceding results gives the lower bound $\sigma_{\ell}$ as in Theorem \ref{thm:main-intro}.

\section*{Acknowledgments}

A.H. was supported by the University of Manitoba's Faculty of Science Undergraduate Summer Research Award and in part by NSERC of Canada Discovery Grant RGPIN-2026-06488. A.P. was supported by NSERC of Canada Discovery Grant RGPIN-2026-06488. C.S. was supported by the Mitacs Globalink Research Internship program. He also thanks Dr.~Yat Ming Chan and Dr.~Jianghua Lu for their valuable teaching in geometry.

\section*{AI contribution disclosure}

Exploration with ChatGPT suggested the key argument in \Cref{sec:lower_new}, which the authors subsequently reconstructed and checked in detail. ChatGPT also assisted with the proof of Proposition~\ref{constantwidth}, specifically with the estimate $\mathrm{diam}(M)\leq w$, and with the choice of the point $z$ in Lemma~\ref{bounddual}. Qwen supplied an initial proof of Proposition~\ref{cond2redundant}. The authors independently validated and finalized all arguments before incorporating them into the manuscript.



\end{document}